%% file: main.tex
\documentclass{article}

\usepackage[colorlinks,linkcolor=blue,citecolor=blue, pagebackref=true,linktocpage]{hyperref}
\usepackage[margin=1.1in]{geometry}
\usepackage[numbers]{natbib}
\usepackage{booktabs}

\input{preamble}

\theoremstyle{plain}

\newtheorem{theorem}{Theorem}[section]
\newtheorem{lemma}[theorem]{Lemma}
\theoremstyle{definition}
\newtheorem{definition}[theorem]{Definition}

\usepackage{authblk}

\title{Distribution-uniform strong laws of large numbers}
\author{
  Ian Waudby-Smith$^\diamond$, Martin Larsson$^\dagger$, and Aaditya Ramdas$^\ddagger$\\
  $^\diamond$Department of Statistics \& The Miller Institute, University of California, Berkeley\\
  $^\dagger$Department of Mathematical Sciences, Carnegie Mellon University\\
  $^\ddagger$Department of Statistics, Stanford University
}
\date{\today}

\begin{document}
\maketitle
\setcounter{tocdepth}{2}
\makeatletter
\renewcommand\tableofcontents{%
  \@starttoc{toc}%
}

\makeatother

\begin{abstract}
\input{abstract}

\end{abstract}

\tableofcontents

\input{content}

\subsection*{Acknowledgments}
The authors thank Morenikeji Neri and anonymous referees for helpful comments and suggestions. The authors acknowledge the use of language models to find errors and inconsistencies in mathematical derivations and notation. Language models were not used to generate any prose or proofs.

\bibliographystyle{plainnat}
\bibliography{references.bib}

\newpage
\appendix
\input{appendix}

\end{document}

%% file: preamble.tex
\usepackage[utf8]{inputenc}
\usepackage{amsmath, amsthm, amssymb, amsfonts}
\usepackage{dsfont}
\usepackage{graphicx}
\usepackage[makeroom]{cancel}
\usepackage[shortlabels]{enumitem}
\usepackage{mathabx}
\usepackage{thmtools}
\usepackage{algorithm}
\usepackage{algorithmic}
\usepackage[capitalise,noabbrev]{cleveref}
\usepackage{xcolor}
\usepackage{array}
\usepackage{colortbl}
\usepackage{tcolorbox}
\usepackage{accents}
\usepackage{mathrsfs}
\usepackage{tikz}
\usetikzlibrary{positioning, arrows.meta}

\usepackage[normalem]{ulem}

\usepackage{autonum}

\declaretheorem[name=Corollary,numberwithin=section]{corollary}
\declaretheorem[name=Proposition,numberwithin=section]{proposition}

\newtheorem*{theorem*}{Theorem}
\newtheorem{remark}{Remark}

\newcommand{\RR}{\mathbb R}
\newcommand{\EE}{\mathbb E}
\newcommand{\UU}{\mathbb U}
\newcommand{\NN}{\mathbb N}

\newcommand{\Fcal}{\mathcal F}

\newcommand{\Pcal}{\mathcal P}

\newcommand{\1}{\mathds{1}}

\newcommand{\seq}[4]{(#1)_{{#2}={#3}}^{#4}}

\newcommand{\infseqt}[1]{\seq{#1}{t}{1}{\infty}}
\newcommand{\infseqn}[1]{\seq{#1}{n}{1}{\infty}}

\newcommand{\as}{{\mathrm{a.s.}}}

\newcommand{\oPcalas}{\widebar o_\Pcal}%
\newcommand{\OPcalas}{\widebar O_\Pcal}

\newcommand{\iid}{\text{i.i.d.}}

\newcommand{\eps}{\varepsilon}
\newcommand{\Var}{\mathrm{Var}}

\newcommand{\leqc}{{\leq c}}
\newcommand{\leqo}{{\leq 1}}

\newcommand{\supP}{{\sup_{P \in \Pcal}}}
\newcommand{\Pin}{{P \in \Pcal}}
\newcommand{\Pstar}{{P^\star}}
\newcommand{\supkm}{{\sup_{k \geq m}}}
\newcommand{\tth}{{\text{th}}}
\newcommand{\mto}{{m \to \infty}}
\newcommand{\Lto}{{L \to \infty}}

\newcommand{\med}{\mathrm{med}}

\newcommand{\Prob}{P}

\newcommand{\brackP}{{(P)}}

\newif\ifverbose %
\verbosefalse %

%% file: abstract.tex
We revisit the question of whether the strong law of large numbers (SLLN) holds uniformly in a rich family of distributions, culminating in a distribution-uniform generalization of the Marcinkiewicz-Zygmund SLLN and settling an open problem posed by Wald. These results can be viewed as extensions of Chung's distribution-uniform SLLN to random variables with uniformly integrable $q^\text{th}$ absolute (central) moments for $0 < q < 2$. Furthermore, we show that uniform integrability of the $q^\text{th}$ moment is both sufficient and necessary for the SLLN to hold uniformly at the Marcinkiewicz-Zygmund rate of $n^{1/q - 1}$, which  also settles the question of necessity that Chung raised in his 1951 paper for $q = 1$. These proofs centrally rely on novel distribution-uniform analogues of some familiar almost sure convergence results including the Khintchine-Kolmogorov convergence theorem, Kolmogorov's three-series theorem, a stochastic generalization of Kronecker's lemma, and the Borel-Cantelli lemmas.
  We also consider the non-identically distributed case
  and provide an application to strongly consistent estimation of variances as well as to M-estimation.

%% file: content.tex
\section{Introduction}\label{section:introduction}

In his 1951 Berkeley Symposium paper titled ``The strong law of large numbers'' \citep{chung_strong_1951}, Kai Lai Chung
writes \textit{``For use in certain statistical applications Professor Wald raised the question of the uniformity of the strong [law of large numbers] with respect to a family of [distributions]''}. Chung's paper proceeds to provide a concrete answer to that question, yielding
a generalization of Kolmogorov's strong law of large numbers (SLLN) that holds uniformly in a rich family of distributions having a uniformly integrable first absolute moment. Let us recall (a minor refinement of) Chung's distribution-uniform SLLN here.

\begin{theorem*}[Chung's $\Pcal$-uniform strong law of large numbers \citep{chung_strong_1951,ruf2022composite}]\label{theorem:chung}
  Let $ \Pcal$ be a collection of probability distributions and $\infseqn{X_n}$ be independent and identically distributed random variables defined on the probability spaces $(\Omega, \Fcal,  \Pcal):= (\Omega, \Fcal, P)_{P \in  \Pcal}$ satisfying the $\Pcal$-uniform integrability ($\Pcal$-UI) condition
  \begin{equation}\label{eq:chung-uniform-integrability}
    \lim_{L \to \infty} \sup_{P \in  \Pcal} \EE_P \left ( |X- \EE_P(X)| \cdot \1{\{ |X - \EE_P(X)| > L \}} \right ) = 0.
  \end{equation}
  Then for every $\eps > 0$,
  \begin{equation}\label{eq:chung-slln}
    \lim_{m \to \infty} \sup_{P \in  \Pcal} \Prob \left ( \sup_{k \geq m} \left \lvert \frac{1}{k} \sum_{i=1}^k X_i - \EE_P(X) \right \rvert \geq \eps \right ) = 0.
  \end{equation}
\end{theorem*}

As perhaps anticipated by Wald, Chung's SLLN has seen several applications in statistics, e.g.~for uniform strong consistency of maximum likelihood estimators in both classical \citep{moran1971uniform} and contemporary \citep{kramlinger2023uniformly} works and prompted generalizations to empirical process theory \citep{sheehy1988uniformity}; we will discuss applications to variance estimation and M-estimation in \cref{section:application}.
Notice that Chung's SLLN recovers Kolmogorov's as a special case when the class of distributions $\Pcal = \{P\}$ is taken to be a singleton such that $\EE_P|X| < \infty$ since for any sequence of random variables $\infseqn{Y_n}$,
\begin{equation}\label{eq:equivalence-almost sure-time-uniform-highprob}
  \Prob \left ( \lim_{n \to \infty} Y_n = 0 \right ) = 1\quad\text{if and only if}\quad\forall \eps > 0,~\lim_{m \to \infty} \Prob \left ( \sup_{k \geq m} |Y_k| \geq \eps \right ) = 0.
\end{equation}
The equivalence in \eqref{eq:equivalence-almost sure-time-uniform-highprob} highlights why Chung's original characterization of the SLLN holding ``$\Pcal$-uniformly'' in \eqref{eq:chung-slln} is a natural one.
Despite Chung's advance, there are four open questions that we aim to address in this paper:
\begin{enumerate}[label = (\roman*)]
\item Can the convergence rate in \eqref{eq:chung-slln} be improved in the presence of higher moments in the sense of \citet{marcinkiewicz1937fonctions}? That is, can it be shown that for all $\eps > 0$,
  \begin{equation}\label{eq:questions-i}
    \lim_{m \to \infty} \sup_\Pin \Prob \left ( \sup_{k \geq m} \left \lvert \frac{1}{k^{1/q}} \sum_{i=1}^k \left ( X_i - \EE_P(X) \right ) \right \rvert \geq \eps \right ) = 0
  \end{equation}
  under certain $\Pcal$-UI conditions on the $q^\tth$ moment for $1 < q< 2$?
\item Is it possible to restrict the \emph{divergence} rate when $X$ has fewer than 1 but more than 0 finite absolute moments, again in the sense of \citet{marcinkiewicz1937fonctions}? That is, can it be shown that
  \begin{equation}\label{eq:questions-ii}
    \lim_{m \to \infty} \sup_\Pin \Prob \left ( \sup_{k \geq m} \left \lvert \frac{1}{k^{1/q}} \sum_{i=1}^k  X_i \right \rvert \geq \eps \right ) = 0,
  \end{equation}
  under similar $\Pcal$-UI conditions but for $0 < q<1$ even when $\EE_P |X| = \infty$ for some $P \in \Pcal$?
\item Are $\Pcal$-UI conditions \emph{necessary} for $\Pcal$-uniform SLLNs to hold (in addition to being sufficient)? That is, if the condition in \eqref{eq:chung-uniform-integrability} does not hold, can it be shown that for some positive constant $C > 0$,
  \begin{equation}\label{eq:questions-iii}
    \lim_{m \to \infty} \sup_\Pin \Prob \left ( \sup_{k \geq m} \left \lvert \frac{1}{k} \sum_{i=1}^k \left ( X_i - \EE_P(X) \right ) \right \rvert \geq C \right ) > 0,
  \end{equation}
  with analogous questions in the case of higher or lower $\Pcal$-UI moments as in (i) and (ii)?
\item Does an analogue of Chung's SLLN exist for independent but \emph{non-identically distributed} random variables, such as in the sense of \citet[\S IX, Theorem 12]{petrov1975laws}? That is, can it be shown that
  \begin{equation}\label{eq:questions-iv}
    \lim_{m \to \infty} \sup_\Pin \Prob \left ( \supkm \left \lvert \frac{1}{a_k} \sum_{i=1}^k (X_i - \EE_P(X_i)) \right \rvert \geq \eps \right ) = 0
  \end{equation}
  for some appropriately chosen sequence $a_n \nearrow \infty$, and if so, under what conditions on $\infseqn{X_n}$?
\end{enumerate}
We provide positive answers to (i), (ii), (iii), and (iv) in Theorems~\ref{theorem:mz-slln}(i), \ref{theorem:mz-slln}(ii), \ref{theorem:mz-slln}(iii), and \ref{theorem:slln-noniid}, respectively.

\begin{remark}[On centered versus uncentered uniform integrability]\label{remark:centered-vs-uncentered-UI}
As outlined by \citet[Remark 4.5]{ruf2022composite}, the assumption displayed in \eqref{eq:chung-uniform-integrability} is a minor refinement of \citet{chung_strong_1951} whose original result made the (stronger) uncentered $\Pcal$-UI assumption,
\begin{equation}\label{eq:uncentered-uniform-integrability}
 \lim_\Lto \sup_\Pin \EE_P \left ( |X| \1 \{|X| > L \} \right ) = 0 
\end{equation}
in place of \eqref{eq:chung-uniform-integrability} but yielding the same conclusion in \eqref{eq:chung-slln}.
While the difference between \eqref{eq:chung-uniform-integrability} and \eqref{eq:uncentered-uniform-integrability} may seem minor --- indeed, for a single $P \in \Pcal$, $\EE_P|X| < \infty$ and $\EE_P|X - \EE_P(X)| < \infty$ are equivalent --- we highlight in \cref{theorem:mz-slln}(iii) how \eqref{eq:chung-uniform-integrability} is both sufficient \emph{and necessary} for the SLLN to hold, while the same cannot be said for \eqref{eq:uncentered-uniform-integrability}, drawing an important distinction between the two.
\end{remark}
\begin{remark}[On the phrase ``uniform integrability'']\label{remark:UI}
Note that the phrase ``uniform integrability'' is commonly used to refer to an analogue of \eqref{eq:chung-uniform-integrability} holding for a family of \emph{random variables} $\infseqn{X_n}$ on the \emph{same} probability space $(\Omega, \Fcal, P)$ (as in \citet[\S 4.5]{chung2001course}, \citet{chong1979theorem,chandra1992cesaro,hu2011note}, and \citet{hu2017note} among others) in the sense that
\begin{equation}
  \lim_\Lto \sup_{k \in \NN} \EE_P \left ( |X_k - \EE_P(X_k)| \cdot \1 \{ |X_k - \EE_P(X_k)| > L \} \right ) = 0,
\end{equation}
while the presentation in \eqref{eq:chung-uniform-integrability} is a statement about a \emph{single} random variable on a \emph{collection} of probability spaces $(\Omega, \Fcal, \Pcal)$ (as also seen in \citet[pp. 93--94]{chow_integration_1997} and \citet[Section 4.2]{ruf2022composite}). Clearly, these two presentations communicate a similar underlying property, but they are used in conceptually different contexts and for this reason, we deliberately write ``$\Pcal$-UI'' to emphasize adherence to \eqref{eq:chung-uniform-integrability} and avoid ambiguity.
\end{remark}
\begin{remark}[On the desire for new proof techniques]
  One might wonder if answers to (i), (ii), and (iv) could have been obtained straightforwardly by simply inspecting the proofs of \citet{marcinkiewicz1937fonctions} and \citet[\S IX, Theorem 12]{petrov1975laws} and carefully keeping track of distribution-dependent constants. Unfortunately, such an approach would not have been immediately successful due to their use of certain qualitative limiting statements (e.g.~the Borel-Cantelli lemmas or Kronecker's lemma). One of the contributions of the present work is in the development of uniform analogues of these technical tools as well as highlighting that the conditions for their application become more subtle in the uniform setting (see \cref{lemma:kronecker} or the caption of \cref{fig:thm1i}). Nevertheless, these uniform tools and their use reduce to their familiar pointwise analogues when the class of distributions is a singleton.

  After the present paper was accepted for publication, we learned of the work of \citet{nguyen2021maximal} who derive a maximal inequality for partial sums that can be used to provide an answer to (i) in the \iid{} case and arrive at \eqref{eq:questions-i}. Their results do not appear relevant for answering (ii)--(iv). Finally, we remark that their derivations do not involve uniform analogues of the tools used in the classical proofs \citep{marcinkiewicz1937fonctions,petrov1975laws} as highlighted in the previous paragraph.
\end{remark}
The discussion surrounding \eqref{eq:equivalence-almost sure-time-uniform-highprob} motivates the following definition which summarizes, extends, and makes succinct Chung's notion of sequences that vanish both $\Pcal$-uniformly and almost surely.

\begin{definition}[Distribution-uniformly and almost surely vanishing sequences]
  \label{definition:dt-uniform}
  Let $\Pcal$ be a collection of distributions and $\infseqn{Y_n}$ be random variables defined on $(\Omega, \Fcal, P)$ for each $P \in \Pcal$. We say that $\infseqn{Y_n}$ \uline{$\Pcal$-uniformly vanishes almost surely} if for any $\eps > 0$,
  \begin{equation}\label{eq:dt-uniform-convergence}
    \lim_{m \to \infty} \sup_{P \in \Pcal} \Prob \left (\sup_{k \geq m} |Y_k| \geq \eps \right ) = 0,
  \end{equation}
  and as a shorthand for \eqref{eq:dt-uniform-convergence}, we write
  \begin{equation}\label{eq:shorthand}
    Y_n = \oPcalas(1).
  \end{equation}
  Moreover, for a positive and monotonic sequence $(r_n)_{n=1}^\infty$, we say that $Y_n = \oPcalas(r_n)$ if $Y_n/r_n = \oPcalas(1)$.
\end{definition}
Clearly, if a sequence is $\Pcal$-uniformly vanishing almost surely, then it is both $\Pcal$-uniformly vanishing in probability for the same class $\Pcal$ as well as vanishing $P$-almost surely for every $P \in \Pcal$. Using the notation of Definition~\ref{definition:dt-uniform}, the desideratum in \eqref{eq:questions-i} can be rewritten as $\frac{1}{n}\sum_{i=1}^n (X_i - \EE(X)) = \oPcalas(n^{1/q-1})$ with similar presentations for \eqref{eq:questions-ii}--\eqref{eq:questions-iv}.

While Definition~\ref{definition:dt-uniform} is sufficient to \emph{state} our main results (presented in Theorems~\ref{theorem:mz-slln} and~\ref{theorem:slln-noniid}), intermediate steps of their proofs centrally rely on generalizing the notion of almost surely \emph{convergent} (but not necessarily vanishing) sequences to a class of distributions $\Pcal$ as well as a notion of distribution-uniform stochastic boundedness. Indeed, classical proofs of SLLNs in the $P$-pointwise setting rely on showing that certain weighted sums are $P$-almost surely convergent (to potentially random quantities depending on $P$), to which the deterministic Kronecker lemma is applied on the same set of $P$-probability one to argue that an appropriate sequence \emph{vanishes}. To facilitate a similar argument uniformly in a class $\Pcal$, we introduce so-called ``$\Pcal$-uniform Cauchy sequences'' as well as $\Pcal$-uniformly stochastically bounded sequences.
\begin{definition}[Distribution-uniform Cauchy sequence]
  \label{definition:distribution-uniform-cauchy-sequence}
  Let $\Pcal$ be a collection of distributions and $\infseqn{Y_n}$ a sequence of random variables defined on $(\Omega, \Fcal, \Pcal)$. We say that $\infseqn{Y_n}$ is a \uline{$\Pcal$-uniform Cauchy sequence} if for any $\eps > 0$,
  \begin{equation}
    \lim_{m \to \infty} \sup_\Pin \Prob \left ( \sup_{k,n \geq m} |Y_k - Y_n| \geq \eps \right ) = 0.
  \end{equation}
\end{definition}
It is easy to check that a $\Pcal$-uniform Cauchy sequence is $P$-almost surely a Cauchy sequence (and hence is $P$-almost surely convergent) for every $P \in \Pcal$. Moreover, if $Y_n - C = \oPcalas(1)$ for any fixed $C \in \RR$, then $\infseqn{Y_n}$ is $\Pcal$-uniformly Cauchy, but the limit point of a $\Pcal$-uniform Cauchy sequence can more generally be random and depend on the individual distributions $P \in \Pcal$. To control the distribution- and time-uniform stochastic boundedness of sequences (which will routinely appear in the context of limit points of $\Pcal$-uniform Cauchy sequences), we introduce the following definition.
\begin{definition}[$\Pcal$-uniformly stochastically bounded]
  \label{definition:stochastic-boundedness}
  Let $\infseqn{Y_n}$ be a sequence of random variables on $(\Omega,\Fcal, \Pcal)$. We say that $\infseqn{Y_n}$ is \uline{$\Pcal$-uniformly stochastically bounded} if for every $\delta > 0$, there exists $B_\delta > 0$ so that
  \begin{equation}\label{eq:stochastic-boundedness}
    \sup_{n \in \NN}\supP \Prob \left ( |Y_n| \geq B_\delta \right ) < \delta,
  \end{equation}
  and as a shorthand for \eqref{eq:stochastic-boundedness}, we write
  \begin{equation}
    Y_n = \OPcalas(1).
  \end{equation}
  For a monotonic and real sequence $\infseqn{r_n}$, we say that $Y_n = \OPcalas(r_n)$ if $Y_n / r_n = \OPcalas(1)$.
\end{definition}
\cref{definition:stochastic-boundedness} should be contrasted with the weaker and more familiar notion of $\Pcal$-uniform \emph{asymptotic} stochastic boundedness which states that for every $\delta > 0$, there exist some $B_\delta > 0$ \emph{and} $N_\delta > 0$ so that for every $n \geq N_\delta$, $\sup_\Pin P(|Y_n| \geq B_\delta) < \delta$. On the other hand, \cref{definition:stochastic-boundedness} can be thought of as the former notion but with $N_\delta = 1$ for all $\delta > 0$. In this paper, we will only be concerned with \cref{definition:stochastic-boundedness} which will play an important role when applying \cref{lemma:kronecker} as the final step in the proofs of Theorems~\ref{theorem:mz-slln} and~\ref{theorem:slln-noniid}.

\subsection{Notation and conventions}

Let us now make explicit some notation and conventions that will be used throughout the paper.
\begin{itemize}
  \item Individual distributions are denoted by the capital letter $P$ and collections of distributions are denoted by calligraphic capital letters (typically $\Pcal$).
    \item We write ``$\Pcal$-UI'' (or simply ``UI'') for ``$\Pcal$-uniformly integrable'' when it is clear from context that the phrase is used as an adjective and ``$\Pcal$-uniform integrability'' when used as a noun.
  \item Collections of probability spaces are written as $(\Omega, \Fcal, \Pcal)$ to denote $(\Omega, \Fcal, P)_{P \in \Pcal}$.
  \item If the $q^\tth$ absolute central moment of $X$ is $\Pcal$-UI, i.e.
    \begin{equation}
      \lim_\Lto \sup_\Pin \EE_P \left ( |X - \EE_P(X)|^q \1 \left \{ |X - \EE_P(X)|^q > L \right \} \right ) = 0,
    \end{equation}
    we condense this to ``the $q^\tth$ moment of $X$ is $\Pcal$-UI'' and omit the qualifiers ``absolute'' and ``central''.
  \item The phrase ``independent and identically distributed'' is abbreviated to ``\iid{}''.
  \item For real numbers $a, b \in \RR$, we use $a \land b$ to denote $\min\{a, b\}$ and $a \lor b$ to denote $\max\{a, b\}$. 
  \item We write $b_n \nearrow \infty$ for a real sequence $\infseqn{b_n}$ if it is nondecreasing and diverging to $\infty$.
  \item We omit the subscript $P$ from $\EE_P(X)$ when using the shorthand notation in \eqref{eq:shorthand}. For example, we write $\frac{1}{n}\sum_{i=1}^n X_i - \EE(X) = \oPcalas(1)$ if in fact
    \begin{equation}
      \forall \eps > 0,\quad\lim_\mto \sup_\Pin  \Prob \left ( \supkm \left \lvert  \frac{1}{k}\sum_{i=1}^k X_i - \EE_P(X) \right \rvert \geq \eps \right )= 0.
    \end{equation}
  \item If $\EE_P|X_n|$ is finite for every $P \in \Pcal$ and every $n \in \NN$, we say that ``the SLLN holds at a rate of $\oPcalas(a_n / n)$'' to mean that $a_n^{-1} \sum_{i=1}^n (X_i - \EE(X_i)) = \oPcalas(1)$ for the sequence $a_n \nearrow \infty$. Similarly, if $\EE_P|X_n| = \infty$ for some $P \in \Pcal$ and some $n \in \NN$, then we use the same phrase ``the SLLN holds at a rate of $\oPcalas(a_n / n)$'' if $a_n^{-1} \sum_{i=1}^n X_i = \oPcalas(1)$. For example, \citet{chung_strong_1951} gives conditions under which the SLLN holds at a rate of $\oPcalas(1)$.
\end{itemize}

\subsection{Outline and summary of contributions}

Below we outline how the paper will proceed, summarizing our main contributions.

\begin{itemize}
\item \cref{section:slln} contains our main results --- Theorems~\ref{theorem:mz-slln}(i), \ref{theorem:mz-slln}(ii), \ref{theorem:mz-slln}(iii), and \ref{theorem:slln-noniid} --- which provide answers to the questions posed in \eqref{eq:questions-i}, \eqref{eq:questions-ii}, \eqref{eq:questions-iii}, and \eqref{eq:questions-iv}, respectively. In short, these theorems show that the SLLN holds at a rate of $\oPcalas(n^{1/q-1})$ in the \iid{} case \emph{if and only if} they have a $\Pcal$-UI $q^\tth$ moment, and that it holds at a rate of $\oPcalas(a_n / n)$ in the non-\iid{} case under certain $\Pcal$-uniform conditions on the sum $\sum_{k=1}^\infty \EE|X_k - \EE X_k |^q / a_k^q$ and its tails.

\item \cref{section:application} illustrates applications of \cref{theorem:mz-slln} to the basic statistical problems of variance estimation and M-estimation in Subsections~\ref{section:variance} and~\ref{section:m-est}. In the former case, \cref{proposition:estimation} appears to be the first result to quantify the rates of uniform convergence of variance estimates. In the latter, \cref{section:m-est} contains results that lift some theorems of \citet[\S 5]{van2000asymptotic} to the distribution-uniform and almost sure regimes.
\item \cref{section:other-strong-laws} contains distribution-uniform analogues of several almost sure convergence results that are commonly used in the proofs of SLLNs. These include analogues of the Khintchine-Kolmogorov convergence theorem (\cref{section:khintchine-kolmogorov}), the Kolmogorov three-series theorem (\cref{section:kolmogorov-three-series}), Kronecker's lemma (\cref{section:kronecker}), and the Borel-Cantelli lemmas (\cref{section:borel-cantelli}). These results rely on the notion of a distribution-uniform Cauchy sequence, whose definition is provided in \cref{definition:distribution-uniform-cauchy-sequence} and which serves as a $\Pcal$-uniform generalization of a sequence that is $P$-almost surely convergent.

\item \cref{section:complete proofs} contains complete proofs to Theorems~\ref{theorem:mz-slln} and~\ref{theorem:slln-noniid}. After considering the ``right'' generalizations of distribution-uniform convergence (in Definitions~\ref{definition:dt-uniform} and~\ref{definition:distribution-uniform-cauchy-sequence}), the high-level structure of the proofs to Theorems~\ref{theorem:mz-slln}(i), \ref{theorem:mz-slln}(ii), and \ref{theorem:slln-noniid} largely mirror those of their $P$-pointwise counterparts due to Kolmogorov, Marcinkiewicz, and Zygmund in the sense that they use analogous technical theorems and lemmas from \cref{section:other-strong-laws} in similar succession. One exception to this is the combination of Kolmogorov's three-series theorem and Kronecker's lemma --- certain subtleties surrounding uniform stochastic boundedness of $\Pcal$-uniform Cauchy sequences requires the introduction of another three-series theorem provided in \cref{theorem:boundedness three series theorem}. Furthermore, our proofs noticeably deviate from their $P$-pointwise counterparts in satisfying the conditions of our $\Pcal$-uniform three series theorems (Theorems~\ref{theorem:kolmogorov-three-series} and~\ref{theorem:boundedness three series theorem}). These require additional care in both cases, relying for example on a few delicate applications of the de la Vall\'ee Poussin criterion of uniform integrability; details can be found in Lemmas~\ref{lemma:sufficient-conditions-puniform-kolmogorov-three-series-iid}--\ref{lemma:satisfying uniform boundedness three series for q smaller than 1}.

\end{itemize}

\section{Distribution-uniform strong laws of large numbers}\label{section:slln}
We begin by presenting our first main result which gives both necessary and sufficient conditions for the SLLN to hold at a rate of $\oPcalas(n^{1/q - 1})$ in the \iid{} setting, providing answers to the questions posed in \eqref{eq:questions-i}, \eqref{eq:questions-ii}, and \eqref{eq:questions-iii}.

\begin{theorem}[$\Pcal$-uniform Marcinkiewicz-Zygmund strong law of large numbers]\label{theorem:mz-slln}
  Let $\infseqn{X_n}$ be independent and identically distributed random variables. Consider the following $\Pcal$-UI condition for some $0 < q < 2$:
  \begin{equation}\label{eq:uniform-integrability-qth-moment}
    \lim_{L \to \infty}\sup_\Pin \EE_P \left ( |X - \mu(P;q)|^{q} \1{\{ |X - \mu(P;q)|^q > L \}} \right ) = 0,
  \end{equation}
  where $\mu(P; q) = \EE_P(X)$ if $1 \leq q < 2$ and $\mu(P; q) = 0$ if $0 < q< 1$.
  \begin{enumerate}[label = (\roman*)]
  \item If \eqref{eq:uniform-integrability-qth-moment} holds with $q \in [1,2)$, then
  \begin{equation}\label{eq:mz-slln-q-in-1,2}
    \frac{1}{n^{1/q}}\sum_{i=1}^n (X_i - \EE(X)) = \oPcalas \left ( 1 \right ).
  \end{equation}
  \item If \eqref{eq:uniform-integrability-qth-moment} holds with $q \in (0,1)$, then
  \begin{equation}
    \frac{1}{n^{1/q}}\sum_{i=1}^n X_i = \oPcalas \left ( 1 \right ).
  \end{equation}
\item If \eqref{eq:uniform-integrability-qth-moment} does \uline{not} hold, then
  \begin{equation}\label{eq:slln-doesnt-hold}
    \frac{1}{n^{1/q}}\sum_{i=1}^n (X_i - \mu(P;q)) \neq \oPcalas \left ( 1 \right ).
  \end{equation}
  \end{enumerate}
  In other words, the $\Pcal$-uniform SLLN holds for the average $\frac{1}{n}\sum_{i=1}^n X_i$ with a rate of $\oPcalas(n^{1/q - 1})$ \uline{if and only if} the $q^\tth$ moment of $X$ is $\Pcal$-UI.
\end{theorem}
In the same way that Chung's $\Pcal$-uniform SLLN for UI first moments generalizes Kolmogorov's $P$-pointwise SLLN for finite first moments, \namecref{theorem:mz-slln}s~\ref{theorem:mz-slln}(i) and \ref{theorem:mz-slln}(ii) generalize the Marcinkiewicz-Zygmund \citep{marcinkiewicz1937fonctions} $P$-pointwise SLLN for finite $q^\tth$ moments when $0<q<2;\ q\neq 1$, painting a fuller picture of sufficiency for $\Pcal$-uniform SLLNs in the \iid{} case.

\begin{figure}[h]
  \begin{center}
    \begin{tikzpicture}[>=Stealth, node distance=0.5cm]
      \node (A) {\cref{theorem:mz-slln}(i)};
      \node (B) [above=of A, align=center] {$\Pcal$-Kronecker's lemma \\ (\cref{lemma:kronecker})};
      \node (C) [above left=of B, align=center] {$\Pcal$-Kolmogorov\\three series theorem \\ (\cref{theorem:kolmogorov-three-series})};
      \node (D) [above right=of B, align=center] {$\Pcal$-boundedness\\three series theorem \\ (\cref{theorem:boundedness three series theorem})};
      \node (F) [above=of C, align=center, text width=6cm] {$\Pcal$-UI of $q^\tth$ moment is sufficient for \cref{theorem:kolmogorov-three-series} \\ (\cref{lemma:sufficient-conditions-puniform-kolmogorov-three-series-iid})};
      \node (G) [above=of D, align=center, text width=6cm] {$\Pcal$-UI of $q^\tth$ moment is sufficient for \cref{theorem:boundedness three series theorem} \\ (\cref{lemma:satisfying uniform boundedness three series for q bigger than 1})};

      \draw[->] (B) -- (A);
      \draw[->] (C) -- (B);
      \draw[->] (D) -- (B);
      \draw[->] (F) -- (C);
      \draw[->] (G) -- (D);
    \end{tikzpicture}
  \end{center}
  \caption{A diagrammatic summary of the theorems and lemmas required to prove \cref{theorem:mz-slln}(i). Note that when $\Pcal = \{P\}$ is a singleton, the proof due to \citet{marcinkiewicz1937fonctions} only involves the left branch (\cref{lemma:sufficient-conditions-puniform-kolmogorov-three-series-iid}$\to$\cref{theorem:kolmogorov-three-series}$\to$\cref{lemma:kronecker}) since the right one is trivially satisfied in that case (more discussion can be found in \cref{section:boundedness-in-probability-iid}). The above structure similarly applies to the proof of \cref{theorem:mz-slln}(ii) but we replace Lemmas~\ref{lemma:sufficient-conditions-puniform-kolmogorov-three-series-iid} and~\ref{lemma:satisfying uniform boundedness three series for q bigger than 1} with Lemmas~\ref{lemma:sufficient-conditions-m-z-qlessthan1} and~\ref{lemma:satisfying uniform boundedness three series for q smaller than 1} for the sake of satisfying the conditions of Theorems~\ref{theorem:kolmogorov-three-series} and \ref{theorem:boundedness three series theorem}, respectively.}
  \label{fig:thm1i}
\end{figure}
Turning to \cref{theorem:mz-slln}(iii), the \emph{necessity} of $\Pcal$-UI appears to be new to the literature even in the case of $q = 1$. In fact, Chung's original paper \citep{chung_strong_1951} studied necessary conditions for the $\Pcal$-uniform SLLN but only considered \emph{uncentered} UI as in \eqref{eq:uncentered-uniform-integrability} which turns out not to be necessary in general. Concretely, he showed that if the SLLN in \eqref{eq:mz-slln-q-in-1,2} holds for $q = 1$ \emph{and} the median of $X$ is uniformly bounded, then the uncentered $\Pcal$-UI condition in \eqref{eq:uncentered-uniform-integrability} holds; in other words, if $\sup_\Pin |\med_P(X)| < \infty$ where $\med_P(X) := \sup\{x : \Prob(X \leq x) \leq 1/2 \}$, then
\begin{equation}
  \frac{1}{n} \sum_{i=1}^n (X_i - \EE(X)) = \oPcalas(1)  \implies \lim_\Lto \sup_\Pin \EE_P(|X| \1\{ |X| > L \}) = 0.
\end{equation}
\citet[Remark 2]{chung_strong_1951} uses a simple counterexample to point out that without uniform boundedness of the medians, uncentered $\Pcal$-UI is \emph{not} necessary. Indeed, letting $\Pcal_\NN := \{P_n : n \in \NN \}$ where $P_n$ is the distribution of $X$ with a point mass at $x = n$, we obviously have that the uniform SLLN holds (since the centered sample average is always 0 with $P$-probability one for all $P \in \Pcal_\NN$) and yet $X$ does not satisfy uncentered $\Pcal_\NN$-uniform integrability. Clearly, this counterexample does not apply to the \emph{centered} uniform integrability condition we are considering in \eqref{eq:uniform-integrability-qth-moment}.

\cref{theorem:mz-slln}(iii) also highlights that uniform \emph{boundedness} of the $q^\tth$ moment is not sufficient for the SLLN to hold $\Pcal$-uniformly at a rate of $o(n^{1/q - 1})$. Going further, by the de la Vall\'ee Poussin criterion for uniform integrability \citep{chong1979theorem}, the SLLN holding uniformly at this rate is \emph{equivalent} to the uniform boundedness of $\EE_P \varphi(|X-\mu(P;q)|^q)$ for some positive and nondecreasing function $\varphi$ growing faster than $x \mapsto x$, i.e.~$\lim_{x \to \infty} \varphi(x) / x = \infty$.
Let us now give rough outlines of the proofs of Theorems~\ref{theorem:mz-slln}(i), \ref{theorem:mz-slln}(ii), and \ref{theorem:mz-slln}(iii) (with a diagrammatic overview of the former displayed in \cref{fig:thm1i}), leaving most technical details for \cref{proof:mz-slln}.

  \begin{proof}[Proof outline of \cref{theorem:mz-slln}(i)]
    Since $q = 1$ corresponds to the SLLN of \citet{chung_strong_1951}, we focus on $1 < q < 2$.
    Similar to classical SLLN proofs, we focus our attention on the weighted random variables $\infseqn{Z_n}$ given by $Z_n := (X_n - \EE_P(X_n)) / n^{1/q}$.
    First, in \cref{theorem:kolmogorov-three-series} we develop a $\Pcal$-uniform analogue of the Kolmogorov three-series theorem which states that if for some $c > 0$,
  \begin{equation}\label{eq:theorem1-proofsketch-threeseries}
    \lim_{m \to \infty}\sup_{P \in \Pcal}\sum_{n=m}^\infty \left \lvert \EE_P Z_n^\leqc \right \rvert = 0,\quad \lim_{m \to \infty}\sup_{P \in \Pcal}\sum_{n=m}^\infty \Var_P Z_n^\leqc = 0, ~~\text{and}~~ \lim_{m \to \infty}\sup_{P \in \Pcal}\sum_{n=m}^\infty \Prob(|Z_n| > c) = 0,
  \end{equation}
  where $Z_n^\leqc := Z_n \1\{ | Z_n | \leq c\}$, then $S_n := \sum_{i=1}^n Z_i$ is a $\Pcal$-uniform Cauchy sequence. 
  Indeed, \cref{lemma:sufficient-conditions-puniform-kolmogorov-three-series-iid} focuses on exploiting $\Pcal$-UI of the $q^\tth$ moment to show that \eqref{eq:theorem1-proofsketch-threeseries} holds with $c = 1$.

  We then introduce the $\Pcal$-uniform stochastic Kronecker lemma (\cref{lemma:kronecker}) which states that if $S_n$ is $\Pcal$-uniformly Cauchy \emph{and} $\Pcal$-uniformly stochastically bounded, then for any $b_n \nearrow \infty$, we have
  \begin{equation}
    \frac{1}{b_n} \sum_{i=1}^n b_i Z_i = \oPcalas(1). 
  \end{equation}
  To apply \cref{lemma:kronecker} to our setting, we show that $S_n$ is $\Pcal$-uniformly Cauchy as a consequence of the three-series theorem discussed above combined with \cref{lemma:sufficient-conditions-puniform-kolmogorov-three-series-iid}, and to show that $S_n$ is $\Pcal$-uniformly stochastically bounded, we introduce another three-series-type theorem in \cref{theorem:boundedness three series theorem} which states that if
  \begin{align}
    &\lim_{B \to \infty}\sup_{P \in \Pcal}\sum_{n=1}^\infty \left \lvert \EE_P \left [ \left ( Z_n / B \right ) \1\{|Z_n/B| \leq 1\} \right ] \right \rvert = 0,\\
    &\lim_{B \to \infty}\sup_{P \in \Pcal}\sum_{n=1}^\infty \Var_P \left [ (Z_n/B) \1\{|Z_n/B| \leq 1\} \right ] = 0,~~\text{and}~~\\
    &\lim_{B \to \infty}\sup_{P \in \Pcal}\sum_{n=1}^\infty \Prob(|Z_n/B| > 1) = 0,
  \end{align}
  then $S_n := \sum_{i=1}^n Z_i $ is $\Pcal$-uniformly stochastically bounded in the sense of \cref{definition:stochastic-boundedness}. \cref{lemma:satisfying uniform boundedness three series for q bigger than 1} indeed shows that the above three series conditions are satisfied as long as $X$ has a $\Pcal$-UI $q^\tth$ moment.
  Taking the sequence $\infseqn{b_n}$ to be given by $b_n = n^{1/q}$ and invoking the $\Pcal$-uniform Kronecker lemma yields the desired result:
  \begin{equation}
    \frac{1}{n^{1/q}} \sum_{i=1}^n (X_i - \EE(X)) = \oPcalas(1),
  \end{equation}
  which completes the proof outline of \cref{theorem:mz-slln}(i).
\end{proof}
\begin{proof}[Proof outline of \cref{theorem:mz-slln}(ii)]
  The proof in the case of $0 < q < 1$ proceeds in the same manner as that of $1 < q < 2$ but instead of \cref{lemma:sufficient-conditions-puniform-kolmogorov-three-series-iid} showing that the three-series conditions in \eqref{eq:theorem1-proofsketch-threeseries} are satisfied for $(X_n - \EE_P(X_n)) / n^{1/q}$, it is \cref{lemma:sufficient-conditions-m-z-qlessthan1} that shows that these conditions are satisfied for $X_n / n^{1/q}$, thereby demonstrating that $\sum_{k=1}^n X_k / k^{1/q}$ is $\Pcal$-uniformly Cauchy. Similarly, rather than using \cref{lemma:satisfying uniform boundedness three series for q bigger than 1} to satisfy the $\Pcal$-uniform stochastic boundedness three series above, we use \cref{lemma:satisfying uniform boundedness three series for q smaller than 1}. Again, invoking the $\Pcal$-uniform stochastic Kronecker lemma yields the desired result.
\end{proof}
\begin{proof}[Proof outline of \cref{theorem:mz-slln}(iii)]
  We will describe the proof outline for the case where $1 \leq q < 2$ but a similar argument goes through for $0 < q < 1$ (with all details provided in \cref{proof:mz-slln}). The proof relies on a $\Pcal$-uniform generalization of the second Borel-Cantelli lemma (\cref{lemma:second-borel-cantelli}) which states that for independent events $\infseqn{E_n}$ in $\Fcal$, if the tails of the sums of $\infseqn{\Prob(E_n)}$ do not uniformly vanish, then the probability of the tails of the \emph{unions} of $\infseqn{E_n}$ do not uniformly vanish; more succinctly:
  \begin{equation}\label{eq:theorem1-proof-outline-second-borel-cantelli}
    0 < \lim_{m \to \infty} \supP \sum_{k=m}^\infty \Prob \left ( E_k \right ) \leq \infty \implies \lim_{m \to \infty} \sup_\Pin \Prob \left ( \bigcup_{k =m}^\infty E_k \right ) > 0.
  \end{equation}
  To make use of this lemma, we highlight that for any $P \in \Pcal$,
  \begin{equation}\label{eq:theorem1-proof-outline-lower-bound-on-crossing-prob}
    \Prob \left ( \sup_{k \geq m-1} \frac{1}{k^{1/q}} |S_{k}| \geq 1/2 \right ) \geq \Prob \left ( \supkm \frac{1}{k^{1/q}} |X - \EE_P(X)| \geq 1 \right )
  \end{equation}
  where $S_k := \sum_{i=1}^k (X_i - \EE_P(X_i))$ are the centered partial sums, and hence once paired with \eqref{eq:theorem1-proof-outline-second-borel-cantelli}, it suffices to show that
  \begin{equation}\label{eq:theorem1-proofoutline-vanishing-tail-of-sum-of-tail-probabilities}
    0 < \lim_\mto \sup_\Pin \sum_{k=m}^\infty \Prob \left ( |X_k - \EE_P(X_k)|^q > k  \right ) \leq \infty.
  \end{equation}
  Indeed by \citet[Theorem 2.1]{hu2017note}, (see also \cref{lemma:uniformly-integrable-iff-uniformly-tail-vanishing} for a self-contained proof), the inequality \eqref{eq:theorem1-proofoutline-vanishing-tail-of-sum-of-tail-probabilities} is \emph{equivalent} to the $\Pcal$-UI condition in \eqref{eq:uniform-integrability-qth-moment} being violated, i.e.~\eqref{eq:theorem1-proofoutline-vanishing-tail-of-sum-of-tail-probabilities} holds if and only if
  \begin{equation}\label{eq:theorem1-proofoutline-uniform-integrability}
    \lim_{L \to \infty}\sup_\Pin \EE_P \left ( |X - \EE_P(X) |^q \1{\{ |X - \EE_P(X)|^q > L \}} \right ) > 0,
  \end{equation}
  which completes the proof outline of \cref{theorem:mz-slln}(iii).
\end{proof}

Let us now consider the setting of independent but \emph{non-identically distributed} random variables. The following theorem serves as a distribution-uniform generalization of the well-known strong law of large numbers for independent random variables (see \citet[\S IX, Theorem 12]{petrov1975laws}).

\begin{theorem}[$\Pcal$-uniform strong law for non-identically distributed random variables]\label{theorem:slln-noniid}
  Let $\infseqn{X_n}$ be independent random variables and suppose that for some $q \in [1,2]$ and some positive $a_n \nearrow \infty$, they satisfy
  \begin{equation}\label{eq:slln-noniid-moment-condition}
    \sup_\Pin \sum_{k=1}^\infty  \frac{\EE_P |X_k - \EE_P X_k|^q}{a_k^q} < \infty~~\text{and}~~\lim_{m \to \infty}\sup_\Pin \sum_{k=m}^\infty  \frac{\EE_P |X_k - \EE_P X_k|^q}{a_k^q} = 0.
  \end{equation}
  Then the SLLN holds $\Pcal$-uniformly at a rate of $\oPcalas(a_n / n)$, meaning for any $\eps > 0$,
  \begin{equation}
    \lim_{m \to \infty} \sup_\Pin \Prob \left ( \supkm \left \lvert \frac{1}{a_k} \sum_{i=1}^k (X_i - \EE_P(X_i)) \right \rvert \geq \eps \right ) = 0.
  \end{equation}
\end{theorem}
\begin{proof}[Proof outline of \cref{theorem:slln-noniid}]
  The proof of \cref{theorem:slln-noniid} is identical to that of \cref{theorem:mz-slln}(i) but instead of using \cref{lemma:sufficient-conditions-puniform-kolmogorov-three-series-iid} and \cref{lemma:satisfying uniform boundedness three series for q bigger than 1} to satisfy the conditions of the $\Pcal$-uniform Kolmogorov and stochastic boundedness three-series theorems (Theorems~\ref{theorem:kolmogorov-three-series} and~\ref{theorem:boundedness three series theorem}), we use different arguments found in Lemmas~\ref{lemma:sufficient-conditions-puniform-kolmogorov-three-series} and \ref{lemma:sufficient-conditions-puniform-boundedness-noniid}, respectively. (In fact, the latter two lemmas are simpler and require much softer arguments.) This completes the proof outline of \cref{theorem:slln-noniid}.
\end{proof}

After some inspection, the reader will notice that when instantiated in the identically distributed setting, \cref{theorem:slln-noniid} does not recover \cref{theorem:mz-slln}, meaning that it cannot attain an SLLN rate as fast as $o(n^{1/q - 1})$ in the presence of only $q \in [1,2)$ $\Pcal$-UI moments. This is not surprising and it directly mirrors the relationship between the $P$-pointwise non-\iid{} SLLNs \citep[\S IX, Theorem 12]{petrov1975laws} and the strong laws of Kolmogorov, Marcinkiewicz, and Zygmund in the \iid{} case. The latter proofs in the $P$-pointwise case (and now ours provided in \cref{proof:mz-slln} for the $\Pcal$-uniform case) all crucially exploit the fact that $\infseqn{X_n}$ are identically distributed. Finally, notice that in the $P$-pointwise case, the two conditions in \eqref{eq:slln-noniid-moment-condition} collapse into one, requiring that $\EE_P|X_k - \EE_P X_k|^q / a_k^q$ is summable.

As alluded to in the proof outlines of Theorems~\ref{theorem:mz-slln} and~\ref{theorem:slln-noniid}, our results rely on $\Pcal$-uniform analogues of several familiar almost sure convergence results. Before presenting these in \cref{section:other-strong-laws}, we discuss an application of \cref{theorem:mz-slln} to strongly uniform consistent estimation in the following section.

\section{Applications to strongly uniform consistent estimation}\label{section:application}
Estimating a parameter from a random sample is a fundamental task in the field of statistics. Indeed, using the language of statistical estimation, Kolmogorov's strong law of large numbers is synonymous with saying that the sample average $\widebar X_n := \frac{1}{n}\sum_{i=1}^n X_i$ is a strongly $P$-consistent estimator of the population mean $\EE_P(X)$ (meaning that the estimator converges to the estimand $P$-a.s.) while the theorems of \citet{marcinkiewicz1937fonctions} show that the \emph{rate} of strong $P$-consistency can be improved to $o_\as(n^{1/q - 1})$ when $\EE_P|X|^q < \infty$ for $1 < q < 2$. Using this same language, the SLLNs of \citet{chung_strong_1951} and \cref{theorem:mz-slln}(i) are simply alternative ways of stating that the sample mean is strongly $\Pcal$-\emph{uniformly} consistent for the mean in a class of distributions $\Pcal$ under the appropriate $\Pcal$-UI conditions. For the above reasons, applications of SLLNs to statistical estimation of the mean are immediate and transparent. In the context of other estimation tasks, however, such applications require more care. In this section, we illustrate how the $\Pcal$-uniform SLLNs of this paper can be applied to two fundamental statistical estimation tasks. We consider variance estimation and M-estimation in \cref{section:variance,section:m-est}, respectively but we anticipate that other applications will be explored in followup works.

\subsection{Variance estimation}\label{section:variance}

Almost-sure variance estimation has important implications in sequential hypothesis testing and time-uniform estimation.\footnote{Here, we are using ``time'' to refer to the sample size as is common in the sequential inference literature \citep{howard_exponential_2018,howard2018uniform}.} For example, consider \citet[Section 2.5]{waudby2021time} who derive coverage guarantees of so-called ``asymptotic confidence sequences''. Let us summarize a simplified version of their goals here. Given a potentially infinite stream of \iid{} observations $X_1, X_2, \dots$ on a single probability space $(\Omega, \Fcal, P)$ and any desired coverage probability $(1-\alpha)$ for $\alpha \in (0, 1)$, they wish to construct a sequence of intervals $[L_k^{(\alpha)}, U_k^{(\alpha)}]_{k=m}^\infty$ --- where $L_k^{(\alpha)}$ and $U_k^{(\alpha)}$ are $\sigma(X_1, \dots, X_k)$-measurable --- so that these intervals cover the mean $\EE_P(X)$ uniformly for all $k \geq m$ with probability tending to $(1-\alpha)$, meaning that
\begin{equation}\label{eq:coverage}
  \lim_{m \to \infty} \Prob \left ( \forall k \geq m,\ \EE_P(X) \in [ L_k^{(\alpha)},  U_k^{(\alpha)} ] \right ) = 1-\alpha.
\end{equation}
The guarantee in \eqref{eq:coverage} can be viewed as a time-uniform analogue of the coverage guarantee enjoyed by large-sample confidence intervals for a fixed sample size. \citet[Section 2.5]{waudby2021time} derive intervals satisfying \eqref{eq:coverage} which rely on the variance $\Var_P(X)$ being strongly $P$-consistently estimated \emph{at a polynomial rate}. That is, their intervals involve an estimator $\widehat \sigma_n^2$ so that $\widehat \sigma_n^2 - \Var_P(X) = o(n^{-\beta})$ $P$-almost surely for some $\beta > 0$, and to achieve this, they let $\widehat \sigma_n^2$ be the sample variance and show via the Marcinkiewicz-Zygmund SLLN that this strong polynomial-rate consistency holds. Notice, however, that \eqref{eq:coverage} is a distribution-\emph{pointwise} convergence result, and hence a distribution-pointwise SLLN is sufficient for their goals. It is plausible that any hope of deriving a distribution-\emph{uniform} analogue of \eqref{eq:coverage} will rely on showing that $\widehat \sigma_n^2$ is strongly $\Pcal$-\emph{uniformly} consistent for the variance in the sense that $\widehat \sigma_n^2 - \Var(X) = \oPcalas(n^{-\beta})$ for some $\beta > 0$. The following corollary shows that such a guarantee follows from \cref{theorem:mz-slln}(i), laying some of the groundwork for potential extensions of \eqref{eq:coverage} to the uniform setting.

\begin{corollary}[$\Pcal$-uniformly strongly consistent variance estimation]\label{proposition:estimation}
  Let $\infseqn{X_n}$ be \iid{} random variables defined on the collection of probability spaces $(\Omega, \Fcal, \Pcal)$. Suppose that the $(2+\delta)^\tth$ moment of $X$ is $\Pcal$-UI for some $\delta \in (0, 2)$ so that
  \begin{equation}
    \lim_\Lto \sup_\Pin \EE_P \left ( |X - \EE_P(X)|^{2+\delta} \1\{ |X - \EE_P(X)|^{2+\delta} > L \} \right ) = 0.
  \end{equation}
  Then, the sample variance $\widehat \sigma_n^2$ is a $\Pcal$-uniformly strongly consistent estimator for the variance at a rate of $\oPcalas(n^{2/(2+\delta) - 1})$, meaning $\widehat \sigma_n^2 - \Var(X) = \oPcalas(n^{2/(2+\delta) - 1})$, or more formally,
  \begin{equation}
    \forall \eps > 0,\quad\lim_\mto \sup_\Pin \Prob \left ( \supkm \left \{ \frac{ \left \lvert \widehat \sigma_k^2 - \Var_P(X) \right \rvert}{k^{2 / (2+\delta) - 1} } \right \}  \geq \eps\right ) = 0.
  \end{equation}
\end{corollary}

\begin{proof}
First, letting $\widebar X_n := \frac{1}{n} \sum_{i=1}^n X_i$, notice that for any fixed $P \in \Pcal$, we can write $\widehat \sigma_n^2$ as
\begin{align}
  \widehat \sigma_n^2 - \Var_P(X) &\equiv \frac{1}{n}\sum_{i=1}^n (X_i - \widebar X_n)^2 - \Var_P(X) \\
  &= \underbrace{\frac{1}{n} \sum_{i=1}^n (X_i - \EE_P(X))^2 - \Var_P(X)}_{(\star)} - \underbrace{(\widebar X_n - \EE_P(X))^2}_{(\dagger)}.
\end{align}
Letting $q := 2+\delta$, note that by \cref{theorem:mz-slln}(i) we have that $\widebar X_n - \EE(X) = \oPcalas(n^{1/\gamma-1})$ for $\gamma := 2 / (2/q + 1)$ since $1 < \gamma < 2$ and in particular, $(\dagger) = \oPcalas(n^{2/q - 1})$. To analyze $(\star)$, we write $Y := X - \EE_P(X)$ and note that it suffices to show that $Y^2$ has an uncentered $\Pcal$-UI $(q/2)^\tth$ moment. Indeed,
\begin{align}
  &\lim_\Lto \sup_\Pin \EE_P \left ( \left \lvert Y^2 \right \rvert^{q/2} \cdot \1 \{ \left \lvert Y^2\right \rvert^{q/2} > L \} \right ) \\
  =\ &\lim_\Lto \sup_\Pin \EE_P \left ( \left \lvert X-\EE_P(X) \right \rvert^{q} \cdot \1 \{ \left \lvert X - \EE_P(X)\right \rvert^{q} > L \} \right ) \\
  =\ &0, \label{eq:variance-proof-uniform-integrability-of-squared}
\end{align}
and hence we have that $(\star) = \oPcalas(n^{2/q - 1})$. Putting this together with the analysis for $(\dagger)$, we have
$\widehat \sigma_n^2 - \Var(X) = \oPcalas(n^{2/q - 1})$, completing the proof.
\end{proof}

\subsection{Strongly consistent M-estimation with applications to quantiles}\label{section:m-est}
A fundamental problem in statistics, distributionally robust optimization, and empirical risk minimization is so-called \emph{M-estimation}; see \citet[\S 5]{van2000asymptotic} for a canonical statistical reference and \citet{blanchet2023statistical} and \citet{blanchet2025distributionally} for connections to distributionally robust optimization and empirical risk minimization. Broadly speaking, the task of M-estimation is to estimate the optimizer of a function using the optimizer of a noisy approximation of that function. Describing the problem in a distribution-uniform sense requires additional care because the optimizer and the function it optimizes may depend on the distribution.

Formally, let $(\Theta, d)$ be a metric space and suppose that for each $\Pin$, $M^\brackP: \Theta \to \RR$ is a function and for each $\Pin$, let $\theta(P) \in \Theta$ so that $(\theta(P))_\Pin$ are uniformly well-separated maximizers of $(M^\brackP)_\Pin$ in the sense that
\begin{equation}\label{eq:well-separated-maximizer}
  \forall \eps > 0 \quad \sup_\Pin \sup_{\substack{\theta \in \Theta: \\ d(\theta, \theta(P)) \geq \eps}} \left \{ M^\brackP(\theta) - M^\brackP(\theta(P)) \right \} < 0. 
\end{equation}
In the problem of M-estimation, it is not assumed that $M^\brackP$ is a known quantity, but rather that one has access to \iid{} observations $X_1, \dots, X_n$ from the distribution $P$ so that $M_n(\cdot)$ is a function of $X_1, \dots, X_n$ --- formally, $M_n : \Omega \times \Theta \to \RR$ is a $\sigma(X_1, \dots, X_n) \otimes \mathcal{B}(\Theta)$-measurable map and $\mathcal B(\Theta)$ is the Borel $\sigma$-algebra on $\Theta$ --- and so that for every $\theta \in \Theta$,
\begin{equation}
  M_n(\theta) \approx M^\brackP(\theta).
\end{equation}
The above approximation is typically thought of as holding in the sense of $P$-pointwise weak convergence \citep[\S 5]{van2000asymptotic} and is justified via the pointwise weak law of large numbers. In what follows, we provide an analogous discussion to \citet[\S 5]{van2000asymptotic} but for strong distribution-uniform convergence.
The following proposition is such an analogue of \citet[Theorem 5.7]{van2000asymptotic}. In what follows, assume that the relevant suprema over $\Theta$ are measurable.
\begin{proposition}\label{proposition:vandervaart-1}
  For every $\Pin$, let $M^\brackP:\Theta \to \RR$ be a map and for each $n \in \NN$, let $M_n$ be a function depending on $X_1, \dots, X_n$. For each $\Pin$, fix the quantity $\theta(P) \in \Theta$ and let $\widehat \theta_n$ be a $\sigma(X_1, \dots, X_n)$-measurable estimator of $\theta(P)$ (the ``M-estimator''). Suppose that the following three conditions hold.
  \begin{enumerate}[label = (\roman*)]
  \item The M-estimator $\widehat \theta_n$ is a strongly distribution-uniform near-maximum of $M_n$ in the sense that
  \begin{equation}
    M_n(\widehat \theta_n) \geq \sup_{\theta' \in \Theta} M_n(\theta') - \oPcalas(1).
  \end{equation}
\item $M_n$ converges $\Theta$-uniformly and strongly distribution-uniformly to $M$ in the sense that
  \begin{equation}
    \forall \eps > 0,\quad \lim_\mto \sup_\Pin P \left ( \supkm \sup_{\theta' \in \Theta} \left \lvert M_k(\theta') - M^\brackP(\theta') \right \rvert \geq \eps \right ) = 0.
  \end{equation}
\item The estimands $(\theta(P))_{\Pin}$ are uniformly well-separated maximizers of $(M^\brackP)_\Pin$ as in \eqref{eq:well-separated-maximizer}.
  \end{enumerate}
  Then we have that $d(\widehat \theta_n, \theta) = \oPcalas(1)$, or more formally,
  \begin{equation}
    \forall \eps > 0,\quad \lim_\mto \sup_\Pin P \left ( \supkm d(\widehat \theta_k, \theta(P)) \geq \eps \right ) = 0.
  \end{equation}
\end{proposition}

\begin{proof}
  Fix $\eps > 0$. By assumption (iii), there exists some $\eta_\eps$ so that for all $\Pin$, $d(\theta, \theta(P)) \geq \eps$ implies that $M^\brackP(\theta) < M^\brackP(\theta(P)) - \eta_\eps$. 
  Write $M(\theta) - M(\widehat \theta_n)$ and note that this difference can be upper-bounded as
  \begin{align}
    M(\theta) - M(\widehat \theta_n) &\leq M_n(\theta) - M(\widehat \theta_n) + \oPcalas(1) \\
    &\leq M_n(\widehat \theta_n) - M(\widehat \theta_n) + \oPcalas(1) \\
                                           &\leq \sup_{\theta' \in \Theta} | M_n(\theta') - M(\theta') | + \oPcalas(1),
  \end{align}
  where the first and second inequalities follow from assumptions (ii) and (i) respectively. The final expression is $\oPcalas(1)$ by another application of assumption (ii).
  In particular, there exists a sequence of random variables $\infseqn{Z_n}$ for which $M(\theta) - M(\widehat \theta_n) \leq Z_n$ and $Z_n = \oPcalas(1)$. Using the fact that $Z_n = \oPcalas(1)$ combined with assumption (iii), it holds that 
  \begin{align}
    \sup_\Pin P \left ( \supkm d(\widehat \theta_k, \theta(P)) \geq \eps \right ) &\leq \sup_\Pin P \left ( \supkm \left [ M^\brackP(\theta(P)) - M^\brackP(\widehat \theta_k) \right ] \geq \eta_\eps \right ) \\
    &\leq \sup_\Pin P \left ( \supkm  Z_k \geq \eta_\eps \right ),
  \end{align}
  and the right-hand side vanishes as $m \to \infty$ since $Z_n = \oPcalas(1)$.
\end{proof}

The above result recovers \citet[Theorem 5.7]{van2000asymptotic} when $\Pcal$ is taken to be a singleton and weak convergence is used in place of strong convergence everywhere. The following proposition is a (simplified) strong distribution-uniform analogue of \citet[Lemma 5.10]{van2000asymptotic} which will be used in an application to strong distribution-uniform consistency of quantiles in \cref{corollary:quantiles}.
\begin{proposition}\label{proposition:vandervaart-5.10}
  Let $\Psi_n$ be a map from $\RR$ to $\RR$ depending on $X_1, \dots, X_n$. For every $\Pin$, let $\Psi^\brackP : \RR \to \RR$ be a function and $\theta(P) \in \RR$ a parameter. Let $\widehat \theta_n$ be a $\sigma(X_1, \dots, X_n)$-measurable estimator. Assume the following:
  \begin{enumerate}[label = (\roman*)]
  \item For every $\delta \in \RR$ it holds that $\Psi_n(\theta + \delta) - \Psi(\theta + \delta) = \oPcalas(1)$.
    \item For every $n \in \NN$, $\Psi_n$ is nondecreasing with $\Psi_n(\widehat \theta_n) = 0$.
    \item For every $\eps > 0$, it holds that
      \begin{equation}
	\sup_\Pin \Psi^\brackP(\theta(P) - \eps) < 0 < \inf_\Pin\Psi^\brackP(\theta(P) + \eps).
      \end{equation}
  \end{enumerate}
  Then $\widehat \theta_n - \theta = \oPcalas(1)$.
\end{proposition}
\begin{proof}
  Fix $\eps > 0$. We have
  \begin{align}
    &\sup_\Pin P \left ( \supkm \lvert \widehat \theta_k - \theta(P) \rvert \geq \eps \right )\\
    \leq\ &\sup_\Pin P \left ( \exists k \geq m : \Psi_k(\theta(P) -\eps) \geq 0 \text{ or } \Psi_k(\theta(P) + \eps) \leq 0 \right ).
  \end{align}
  By assumptions (i) and (iii), we have that the above right-hand side vanishes, completing the proof.
\end{proof}
\begin{corollary}\label{corollary:quantiles}
  Let $\Theta \subseteq \RR$. Let $\infseqn{X_n}$ be \iid{} random variables and $\Pcal$ an arbitrary collection of continuous distributions.
  Fix a desired probability $\alpha \in (0, 1)$.
  For each $\Pin$, consider the function $\Psi^{\brackP}(\theta) = P  ( X_1 \leq \theta ) - \alpha$ and let $\theta(P)$ be the unique zero of $\Psi^\brackP$, and hence the $\alpha^\tth$ quantile of $X_1$ under $P$. Suppose that for every $\eps > 0$,
  \begin{equation}
    \sup_{\Pin} \Psi^\brackP(\theta(P) - \eps) < 0 < \inf_\Pin \Psi^\brackP(\theta(P) + \eps).
  \end{equation}
  Define $\Psi_n(\theta)$ as
  \begin{equation}
   \Psi_n(\theta) = \frac{1}{n} \sum_{i=1}^n \left [ \1 \{ X_i \leq \theta \}  - \alpha \right ] 
  \end{equation}
  and let $\widehat \theta_n := \inf \{ \theta' : \Psi_n(\theta') \geq 0 \}$ be the (upper) zero of $\Psi_n$. Then $\widehat \theta_n - \theta = \oPcalas(1)$.
\end{corollary}
\begin{proof}
  Assumptions (ii) and (iii) of \cref{proposition:vandervaart-5.10} are easy to verify. Turning to (i), we have that for any $\Pin$ and $\delta \in \RR$,
  \begin{equation}
    \Psi_n(\theta(P) + \delta) - \Psi^\brackP(\theta(P) + \delta) = \frac{1}{n} \sum_{i=1}^n \1 \{ X_i \leq \theta(P) + \delta \} - P(X_i \leq \theta(P) + \delta).
  \end{equation}
  Applying \cref{theorem:mz-slln}, the above is $\oPcalas(n^{1/q - 1})$ for every $q \in [1,2)$ since for any $\theta' \in \Theta$, the random variable $\1 \{ X_i \leq \theta' \} - P(X_i \leq \theta')$ is mean-zero with a uniformly bounded second moment.
\end{proof}
The above corollary extends the discussion in \citet[Example 5.11]{van2000asymptotic} on median estimation to arbitrary $\alpha^\tth$ quantiles for $\alpha \in (0, 1)$. Furthermore, it establishes that under no additional assumptions to those made in \citet[Example 5.11]{van2000asymptotic}, quantiles converge both distribution-uniformly and almost surely.

\section{Other distribution-uniform convergence results}\label{section:other-strong-laws}

In this section, we provide uniform analogues of various almost sure convergence results including the Khintchine-Kolmogorov convergence theorem, the Kolmogorov three-series theorem, and a stochastic generalization of Kronecker's lemma. These are instrumental to the proofs of Theorems~\ref{theorem:mz-slln} and~\ref{theorem:slln-noniid}.

Note that the classical ($P$-pointwise) forms of these results are stated (either in their assumptions or in their conclusions) in terms of a sequence of random variables \emph{converging} $P$-almost surely, and hence we will rely heavily on the notion of $\Pcal$-uniform Cauchy sequences provided in Definition~\ref{definition:distribution-uniform-cauchy-sequence}.
We begin in the following section with a $\Pcal$-uniform generalization of the Khintchine-Kolmogorov convergence theorem. 
\subsection{A distribution-uniform Khintchine-Kolmogorov convergence theorem}\label{section:khintchine-kolmogorov}
In the classical $P$-pointwise case, the Khintchine-Kolmogorov convergence theorem states that for a sequence of independent mean-zero random variables $\infseqn{X_n}$, if the sum of their variances is finite, i.e.
\begin{equation}\label{eq:khintchine-kolmogorov-pointwise-condition}
    \sum_{k=1}^\infty \Var_P(X_k) < \infty,
\end{equation}
then $\sum_{k=1}^\infty X_k$ is $P$-almost surely finite. 
With the notion of an $\Pcal$-uniform Cauchy sequence in mind, we are ready to state and prove an $\Pcal$-uniform generalization of the Khintchine-Kolmogorov convergence theorem, establishing that the sum $\sum_{k=1}^\infty X_k$ is $\Pcal$-uniformly Cauchy whenever the series in \eqref{eq:khintchine-kolmogorov-pointwise-condition} has $\Pcal$-uniformly vanishing tails.
\begin{theorem}[$\Pcal$-uniform Khintchine-Kolmogorov convergence theorem]\label{theorem:khintchine-kolmogorov}
  For each $\Pin$, let $\infseqn{X_n}$ be independent random variables on $(\Omega, \Fcal, \Prob)$. If
  \begin{equation}\label{eq:petrovlemma8-statement-uniform-tail-decay}
\lim_{m \to \infty}\sup_{\Pin}\sum_{k=m}^{\infty} \Var_{\Prob} (X_k) = 0 
  \end{equation}
  then $S_n := \sum_{i=1}^n (X_i - \EE_{\Prob}(X_i))$ is $\Pcal$-uniformly Cauchy.
\end{theorem}
\begin{proof}
 First note that for any $m \geq 1$, we have that
  \begin{equation}
    \left \{ \sup_{k,n \geq m} |S_k - S_n| \geq \eps \right \} \subseteq \left \{ \sup_{k \geq m} |S_k - S_m| \geq \eps / 2 \right \} \cup  \left \{\sup_{n \geq m} |S_n - S_m| \geq \eps/2 \right \}
  \end{equation}
  and hence for any $\Pin$, Kolmogorov's inequality yields
  \begin{align}
    \Prob \left ( \sup_{k, n \geq m} |S_k - S_n| \geq \eps \right ) &\leq \Prob \left ( \sup_{k \geq m} |S_k - S_m| \geq \eps/2 \right ) + \Prob \left ( \sup_{n \geq m} |S_n - S_m| \geq \eps/2 \right )\\
    &\leq \frac{8}{\eps^2} \cdot \sum_{k=m+1}^\infty \Var_{\Prob}(X_k).
  \end{align}
  Taking suprema over $\Pin$ and limits as $m \to \infty$ and noting the condition in \eqref{eq:petrovlemma8-statement-uniform-tail-decay} yields the desired result, completing the proof.
\end{proof}

\subsection{A distribution-uniform Kolmogorov three-series theorem}\label{section:kolmogorov-three-series}
Now that we have a uniform Khintchine-Kolmogorov convergence theorem, we will use it to prove a uniform analogue of Kolmogorov's three-series theorem. To begin, define the truncated version $X^\leqc$ of a random variable $X$ at a constant $c$ as
\begin{equation}
  X^\leqc := X \cdot \1\{|X| \leq c\}.
\end{equation}
In the $P$-pointwise case, recall that Kolmogorov's three-series theorem states that if the following three series are convergent for some $c > 0$:
\begin{equation}\label{eq:kolmogorov-three-series-pointwise-condition}
  \sum_{n=1}^\infty \EE_P X_n^\leqc,\quad \sum_{n=1}^\infty \Var_P X_n^\leqc,~~\text{and}~~ \sum_{n=1}^\infty \Prob(|X_n| > c),
\end{equation}
then $\sum_{k=1}^\infty X_k$ is $P$-almost surely convergent. Similarly to \cref{theorem:khintchine-kolmogorov} in the previous section, our $\Pcal$-uniform analogue of Kolmogorov's three series theorem will conclude that $\sum_{k=1}^\infty X_k$ is $\Pcal$-uniformly Cauchy as long as the tails of a certain three series are $\Pcal$-uniformly vanishing. 
\begin{theorem}[$\Pcal$-uniform Kolmogorov three-series theorem]\label{theorem:kolmogorov-three-series}
  Let $\infseqn{X_n}$ be a sequence of independent random variables. Suppose that the following three summation tails decay $\Pcal$-uniformly for some $c > 0$:
  \begin{equation}
    \lim_{m \to \infty}\supP\sum_{n=m}^\infty \left \lvert \EE_\Prob X_n^\leqc \right \rvert = 0,\quad \lim_{m \to \infty}\supP\sum_{n=m}^\infty \Var_\Prob X_n^\leqc = 0,~~\text{and}~~ \lim_{m \to \infty}\supP\sum_{n=m}^\infty \Prob(|X_n| > c) = 0.\nonumber
  \end{equation}
  Then $S_n := \sum_{i=1}^n X_i$ is $\Pcal$-uniformly Cauchy.
\end{theorem}
Notice that the first series in \eqref{eq:kolmogorov-three-series-pointwise-condition} does not have an exact analogue in \cref{theorem:kolmogorov-three-series} since the former is not a sum of \emph{absolute values} of $(\EE_\Prob X_n^\leqc)_{n=1}^\infty$ while that of the latter is. In particular, \cref{theorem:kolmogorov-three-series} is not a \emph{strict} generalization of Kolmogorov's three-series theorem in general, but this distinction is inconsequential for the sake of proving ($\Pcal$-uniform or $P$-pointwise) SLLNs, at least in the \iid{} and independent but non-\iid{} settings considered by Kolmogorov, Marcinkiewicz, and Zygmund, as well as \citet[Theorem 12]{petrov1975laws}. Indeed, all of their (and our) proofs ultimately upper bound $\EE_\Prob X_n^\leqc$ for a mean-zero $X_n$ by $\EE_\Prob(|X_n| \cdot \1\{|X_n| \leq c\})$ or by $\EE_\Prob(|X_n| \cdot \1\{|X_n| > c\})$, and hence one can simply analyze $|\EE_\Prob X_n^\leqc|$ from the outset. Detailed discussions and proofs can be found in \cref{section:sufficient-conditions-puniform-kolmogorov-three-series-iid}. Let us now return to and prove \cref{theorem:kolmogorov-three-series}.
\begin{proof}
  Abusing notation slightly, let $S_n^\leqc := \sum_{i=1}^n X_i^\leqc$. Note that for any $m \geq 1$ and any $\eps > 0$, we have 
  \begin{align}
    &\sup_\Pin \Prob \left ( \sup_{n, k \geq m}\left \lvert S_{n} - S_k  \right \rvert \geq \eps \right )\\
    \ifverbose %
    =\ &\sup_\Pin \Prob \left ( \sup_{n \geq k \geq m}\left \lvert S_{n} - S_k  \right \rvert \geq \eps \right )\\
    \leq\ &\sup_\Pin \Prob \left ( \sup_{n \geq k \geq m}\left \lvert S_{n}^\leqc - S_k^\leqc  \right \rvert \geq \eps \right ) + \sup_\Pin \Prob \left ( \exists k \geq m : X_k \neq X_k^\leqc \right )\\
    \fi %
    \leq\ & \sup_\Pin \Prob \left ( \sup_{n \geq k \geq m}\left \lvert S_{n}^\leqc - S_k^\leqc  \right \rvert \geq \eps \right ) +\sup_\Pin \sum_{k=m}^\infty\Prob \left ( |X_k| > c \right ).
  \end{align}
  The second term above vanishes asymptotically by the third series, so it suffices to show that the first term goes to 0 as $m \to \infty$. Indeed,
  \begin{align}
    &\sup_\Pin \Prob \left ( \sup_{n \geq k \geq m}\left \lvert S_{n}^\leqc - S_k^\leqc  \right \rvert \geq \eps \right ) \\
    \ifverbose %
    =\ &\sup_\Pin \Prob \left ( \sup_{n \geq k \geq m}\left \lvert \sum_{i=1}^n X_i\1 \{ |X_i| \leq c \} - \sum_{i=1}^k X_i \1 \{ |X_i| \leq c\}  \right \rvert \geq \eps \right )\\
    =\ &\sup_\Pin \Prob \left ( \sup_{n \geq k \geq m}\left \lvert S_n^\leqc - \EE_\Prob S_n^\leqc + \EE_\Prob S_n^\leqc - \EE_\Prob S_k^\leqc + \EE_\Prob S_k^\leqc - S_k^\leqc  \right \rvert \geq \eps \right )\\
    \fi %
    \leq\ &\underbrace{\sup_\Pin \Prob \left ( \sup_{n \geq k \geq m}\left \lvert (S_n^\leqc - \EE_\Prob S_n^\leqc) - (S_k^\leqc - \EE_\Prob(S_k^\leqc)) \right \rvert \geq \eps/2 \right )}_{(\star)} +\\
    & \underbrace{\sup_\Pin \1 \left \{ \sup_{k \geq n \geq m} |\EE_\Prob S_n^\leqc - \EE_\Prob S_k^\leqc| \geq \eps/2 \right \}}_{(\dagger)}.
  \end{align}
 Now, $(\star) \to 0$ by the second series combined with the $\Pcal$-uniform Khintchine-Kolmogorov convergence theorem (\cref{theorem:khintchine-kolmogorov}). Turning to $(\dagger)$, we have
  \begin{align}
    &\sup_\Pin \1 \left \{ \sup_{k \geq n \geq m} \left \lvert \EE_\Prob S_n^\leqc - \EE_\Prob S_k^\leqc \right \rvert \geq \eps / 2 \right \} \\
    \ifverbose %
    \leq\ &\sup_\Pin \1 \left \{ \sup_{k \geq m} \left \lvert \EE_\Prob S_k^\leqc - \EE_\Prob S_m^\leqc \right \rvert \geq \eps / 4 \right \} + \sup_\Pin \1 \left \{ \sup_{n \geq m} \left \lvert \EE_\Prob S_n^\leqc - \EE_\Prob S_m^\leqc \right \rvert \geq \eps / 4 \right \} \\
    =\ &\sup_\Pin \1 \left \{ \sup_{k \geq m} \left \lvert \sum_{i=m}^k \EE_\Prob X_i^\leqc \right \rvert \geq \eps / 4 \right \} + \sup_\Pin \1 \left \{ \sup_{n \geq m} \left \lvert \sum_{i=m}^n \EE_\Prob X_i^\leqc \right \rvert \geq \eps / 4 \right \}  \\
    \leq\ &\sup_\Pin \1 \left \{ \sup_{k \geq m} \sum_{i=m}^k \left \lvert \EE_\Prob X_i^\leqc \right \rvert \geq \eps / 4 \right \} + \sup_\Pin \1 \left \{ \sup_{n \geq m}  \sum_{i=m}^n \left \lvert\EE_\Prob X_i^\leqc \right \rvert \geq \eps / 4 \right \}  \\
    \fi %
    \leq\ &2\sup_\Pin \1 \left \{ \sum_{i=m}^\infty \left \lvert \EE_\Prob X_i^\leqc \right \rvert \geq \eps / 4 \right \}  
  \end{align}
  which is zero for sufficiently large $m$ by the first of the three series, completing the proof.
\end{proof}

\subsection{A three-series theorem for stochastic boundedness}\label{section:boundedness-in-probability-iid}
In addition to certain partial sums being $\Pcal$-uniformly Cauchy, an important condition that will appear throughout our proofs --- namely in the application of the uniform Kronecker lemma (Lemma~\ref{lemma:kronecker}) --- is that of $\Pcal$-uniform stochastic boundedness (see \cref{definition:stochastic-boundedness}).
Here, we provide a three-series theorem providing sufficient conditions for partial sums to be $\Pcal$-uniformly stochastically bounded and whose conditions are similar in spirit to those of \cref{theorem:kolmogorov-three-series}.

\begin{theorem}[A three-series theorem for $\Pcal$-uniform stochastic boundedness]\label{theorem:boundedness three series theorem}
  Let $\infseqn{Z_n}$ be independent random variables on the probability spaces $(\Omega, \Fcal, \Pcal)$.
  Define the truncated random variables $\infseqn{Z_{n,B}^\leqo}$ given by
  \begin{equation}
    Z_{n,B}^\leqo := (Z_n / B) \cdot \1\{ |Z_n / B| \leq 1 \}.
  \end{equation}
  Suppose that the following three series uniformly vanish as $B \to \infty$:
  \begin{equation}
    \lim_{B \to \infty}\sup_{P \in \Pcal}\sum_{n=1}^\infty \left \lvert \EE_\Prob  Z_{n,B}^\leqo \right \rvert = 0,\quad\lim_{B \to \infty}\sup_{P \in \Pcal}\sum_{n=1}^\infty \Var_P  Z_{n,B}^\leqo = 0,\quad\text{and}\quad \lim_{B \to \infty}\sup_{P \in \Pcal}\sum_{n=1}^\infty \Prob(|Z_n| > B) = 0.\nonumber
  \end{equation}
  Then, $S_n := \sum_{k=1}^n Z_k$ is $\Pcal$--uniformly stochastically bounded, meaning that for any $\delta > 0$, there exists $B_\delta > 0$ so that
  \begin{equation}\label{eq:boundedness three series desideratum}
    \sup_{n \in \NN}\sup_\Pin \Prob \left ( |S_n| \geq B_\delta \right ) < \delta.
  \end{equation}
\end{theorem}
\begin{proof}[Proof of \cref{theorem:boundedness three series theorem}]
  Let $\delta > 0$ be arbitrary. Our goal is to show that there exists $B_\delta$ large enough so that for all $n \geq 1$, \eqref{eq:boundedness three series desideratum} holds.
  Notice that for any $B > 0$, we have that
  \begin{align}
    &\sup_\Pin \Prob \left ( |S_n| > B \right ) \\
    =\ &\sup_\Pin \Prob \left ( \left \lvert \sum_{k=1}^n Z_k \right \rvert  > B \right )\\
    \leq\ & \underbrace{\sup_\Pin \Prob \left ( \left \lvert \sum_{k=1}^n (Z_k / B)\1 \{ |Z_k/ B| \leq 1 \} \right \rvert  > 1 \right )}_{(\star)} + \underbrace{\sup_\Pin  \sum_{k=1}^n \Prob \left ( |Z_k / B| > 1 \right )}_{(\dagger)}.
  \end{align}
  Letting $S_{n,B}^\leqo := \sum_{k=1}^n Z_{k,B}^\leqo$ be the partial sums of the truncated random variables $Z_{k,B}^\leqo := (Z_k/B) \cdot \1\{ |Z_k/B| \leq 1 \}$, notice that we can write $(\star)$ as
  \begin{align}
    &\sup_\Pin \Prob \left ( \left \lvert \sum_{k=1}^n (Z_k / B)\cdot \1\{|Z_k/B| \leq 1\} \right \rvert  > 1 \right ) \\
    =\ &\sup_\Pin \Prob \left ( \left \lvert S_{n,B}^\leq - \EE_\Prob S_{n,B}^\leq + \EE_\Prob S_{n,B}^\leq \right \rvert  > 1 \right )\\
    \leq\ &\sup_\Pin \Prob \left ( \left \lvert S_{n,B}^\leq - \EE_\Prob S_{n,B}^\leq\right \rvert  > 1/ 2 \right ) + \sup_\Pin \Prob \left ( \left \lvert \EE_\Prob S_{n,B}^\leq\right \rvert  > 1/ 2 \right )\\
    =\ &\underbrace{\sup_\Pin \Prob \left ( \left \lvert S_{n,B}^\leq - \EE_\Prob S_{n,B}^\leq\right \rvert  > 1/ 2 \right )}_{(\star i)} + \underbrace{\sup_\Pin \1 \left \{ \left \lvert \EE_\Prob S_{n,B}^\leq\right \rvert  > 1/ 2 \right \}}_{(\star ii)}.
  \end{align}
  By Kolmogorov's inequality, we have that
  \begin{align}
    (\star i) %
    &\leq 4\sup_\Pin \sum_{k=1}^\infty \Var_P \left ( (Z_k / B) \cdot \1\{ |Z_k / B| \leq 1 \} \right ).
  \end{align}
  Furthermore, by the triangle inequality and upper bounding the finite sum by an infinite one, we have
  \begin{equation}
    (\star ii) \leq \sup_\Pin \1 \left \{ \sum_{k=1}^\infty \left \lvert \EE_\Prob \left [ (Z_k / B) \1 \{ |Z_k / B| \leq 1 \} \right ] \right \rvert > 1/2 \right \}.
  \end{equation}
  Once again upper bounding a finite sum by an infinite one, we have
  \begin{equation}
    (\dagger) \leq \sup_\Pin \sum_{k=1}^\infty \Prob \left ( |Z_k/B| > 1 \right ).
  \end{equation}
  Therefore, using the first, second, and third series conditions, we can find $B_\delta > 0$ so that for all $n \geq 1$, we have $(\star i) < \delta/2$, $(\star ii) = 0$, and $(\dagger) < \delta/2$, respectively, and thus
 \begin{align}
   \sup_\Pin \Prob \left ( |S_n| \geq B_\delta \right ) < \delta,
 \end{align} 
 completing the proof.
\end{proof}

\subsection{A distribution-uniform stochastic generalization of Kronecker's lemma}\label{section:kronecker}
In the classical $P$-pointwise setting, proofs of strong laws of large numbers rely on a (non-stochastic) convergence result known as \emph{Kronecker's lemma} which states that if $\infseqn{x_n}$ is a sequence of real numbers so that $\sum_{i=1}^\infty x_i = \ell \in \RR$,
then for any positive sequence $b_n \nearrow \infty$,
\begin{equation}\label{eq:kronecker-introduction-conclusion}
  \lim_{n \to \infty} \frac{1}{b_n}\sum_{i=1}^n b_i x_i = 0.
\end{equation}
This lemma is typically used as follows (consider the Marcinkiewicz-Zygmund SLLN with $1 < q < 2$ for the sake of example). One first shows via the $P$-pointwise Kolmogorov three-series theorem that the sum
\begin{equation}\label{eq:kronecker-introduction-sum}
  \sum_{k=1}^n \frac{X_k - \EE_\Prob(X)}{k^{1/q}}
\end{equation}
is $P$-almost surely convergent as $n \to \infty$, at which point one applies Kronecker's lemma (on the same set of $P$-probability 1) to justify that
\begin{equation}\label{eq:kronecker-introduction-a.s.-convergence}
  \Prob \left (\lim_{n \to \infty}\frac{1}{n^{1/q}} \sum_{i=1}^n (X_i - \EE_\Prob(X)) = 0 \right ) = 1.
\end{equation}
However, it is not clear how Kronecker's lemma can be used to derive a $\Pcal$-uniform analogue of \eqref{eq:kronecker-introduction-a.s.-convergence} if \eqref{eq:kronecker-introduction-sum} is only shown to be $\Pcal$-uniformly Cauchy and especially if the limiting value $\ell \equiv \ell(P)$ of \eqref{eq:kronecker-introduction-sum} is a potentially random quantity whose behavior depends on the distribution $P \in \Pcal$ itself. Indeed, for the $\Pcal$-uniform case, we introduce an additional uniform stochastic boundedness condition given in \cref{definition:stochastic-boundedness}. Satisfying \cref{definition:stochastic-boundedness} in pursuit of proving Theorems~\ref{theorem:mz-slln}(i), \ref{theorem:mz-slln}(ii), and \ref{theorem:slln-noniid} requires additional care (the details of which can be found in \cref{section:boundedness-in-probability-iid}), while this subtlety is easily sidestepped in the $P$-pointwise setting. Nevertheless, the following lemma serves as a stochastic and $\Pcal$-uniform generalization of Kronecker's lemma that lends itself naturally to our goals and reduces to the usual $P$-almost sure application of Kronecker's lemma when $\Pcal = \{P\}$ is a singleton.

\begin{lemma}[A $\Pcal$-uniform stochastic generalization of Kronecker's lemma]\label{lemma:kronecker}
  Let $\infseqn{Z_n}$ be a sequence of random variables so that their partial sums $S_n := \sum_{i=1}^n Z_i$ form a $\Pcal$-uniform Cauchy sequence in the sense of \cref{definition:distribution-uniform-cauchy-sequence} and which is $\Pcal$-uniformly stochastically bounded in the sense of \cref{definition:stochastic-boundedness}.
  Let $b_n \nearrow \infty$ be a positive, nondecreasing, and diverging sequence.
  Then, $b_n^{-1}\sum_{i=1}^n b_i Z_i$ vanishes $\Pcal$-uniformly in the sense of \cref{definition:dt-uniform}.
\end{lemma}

\begin{proof}
  Fix any $\eps > 0$ and any $\delta > 0$. Our goal is to show that for all $m$ sufficiently large,
  \begin{equation}
    \sup_{P \in \Pcal} \Prob \left ( \sup_{k \geq m} \left \lvert \frac{1}{b_k} \sum_{i=1}^k b_i Z_i  \right \rvert \geq \eps \right ) < 4\delta,
  \end{equation}
  where the factor of 4 is only included for mathematical convenience later on.
  Using the assumption that $S_n$ is $\Pcal$-uniformly Cauchy and stochastically bounded, let $B > 0$ and $N \in \NN$ be sufficiently large so that for any $m \geq N$,
  \begin{equation}\label{eq:kronecker-proof-condition-cauchy}
    \sup_{P \in \Pcal} \Prob \left ( \sup_{k,n \geq m} \left \lvert S_n - S_k \right \rvert \geq \eps/6 \right ) < \delta,
  \end{equation}
  and so that
  \begin{equation}\label{eq:kronecker-proof-condition-stochastic-boundedness}
    \sup_\Pin \Prob \left ( |S_m| \geq B \right ) < \delta.
  \end{equation}
  Again using stochastic boundedness and the assumption that $b_n \nearrow \infty$, let $N^\star \equiv N^\star(\eps, B, N) \geq N$ be sufficiently large so that
  \begin{equation}\label{eq:kronecker-proof-condition-increasing-b_m}
    \frac{\eps b_{N^\star}}{6b_N}\geq B,
  \end{equation}
  and so that
  \begin{equation}\label{eq:kronecker-proof-condition-stochastic-boundedness-2}
      \sup_{P \in\Pcal}\Prob \left ( \sum_{i=1}^{N-1} \left \lvert S_i \right \rvert \geq \frac{\eps b_{N^\star}}{6b_{N}} \right ) < \delta,
  \end{equation}
  where we can impose the latter condition since $\sum_{i=1}^{N-1} |S_i|$ is $\Pcal$-uniformly bounded in probability for any fixed $N$ and we can take $\eps b_{N^\star} / 6b_N$ to be arbitrarily large for any fixed $N$ and $\eps$.
  Then for all $m \geq N^\star$,
  \begin{align}
    & \sup_{P \in \Pcal}\Prob \left ( \sup_{k \geq m} \left \lvert  \frac{1}{b_k}\sum_{i=1}^k b_i Z_i  \right \rvert \geq \eps  \right )\\
    =\ &  \sup_{P \in \Pcal}\Prob \left ( \sup_{k \geq m} \left \lvert  S_k - \frac{1}{b_k}\sum_{i=1}^{k-1} (b_{i+1} - b_i) S_i \right \rvert \geq \eps \right ) \label{eq:kronecker-proof-summation-by-parts}\\
    \ifverbose %
    =\ &  \sup_{P \in \Pcal}\Prob \left ( \sup_{k \geq m} \left \lvert  S_k - \frac{1}{b_k}\sum_{i=1}^{N-1} (b_{i+1} - b_i) S_i - \frac{1}{b_k} \sum_{i=N}^{k-1} (b_{i+1} - b_i)S_i \right \rvert \geq \eps \right ) \label{eq:kronecker-proof-breaking-sum-up} \\
    =\ &  \sup_{P \in \Pcal}\Prob \left ( \sup_{k \geq m} \left \lvert  S_k - \frac{1}{b_k}\sum_{i=1}^{N-1} (b_{i+1} - b_i) S_i - \frac{1}{b_k} \sum_{i=N}^{k-1} (b_{i+1} - b_i)S_m -\frac{1}{b_k} \sum_{i=N}^{k-1} (b_{i+1} - b_i)(S_i - S_m) \right \rvert \geq \eps \right ) \\
    =\ &  \sup_{P \in \Pcal}\Prob \left ( \sup_{k \geq m} \left \lvert  S_k - \frac{1}{b_k}\sum_{i=1}^{N-1} (b_{i+1} - b_i) S_i - \frac{b_k-b_N}{b_k} \cdot S_m  -\frac{1}{b_k} \sum_{i=N}^{k-1} (b_{i+1} - b_i)(S_i - S_m) \right \rvert \geq \eps \right ) \label{eq:kronecker-proof-consolidating-after-adding-and-subtracting-S_m} \\
    \fi %
    \leq\ &  \sup_{P \in \Pcal}\Prob \left ( \sup_{k \geq m} \left \lvert  S_k -\frac{b_k-b_N}{b_k} S_m \right \rvert \geq \eps/3 \right ) + \label{eq:kronecker-proof-first-term}\\
    &\sup_{P \in \Pcal}\Prob \left ( \sup_{k \geq m} \left \lvert \frac{1}{b_k}\sum_{i=1}^{N-1} (b_{i+1} - b_i) S_i \right \rvert \geq \eps/3 \right ) + \label{eq:kronecker-proof-second-term}\\
    &\sup_{P \in \Pcal}\Prob \left ( \sup_{k \geq m} \left \lvert \frac{1}{b_k} \sum_{i=N}^{k-1} (b_{i+1} - b_i)(S_i - S_m) \right \rvert \geq \eps/3 \right ),\label{eq:kronecker-proof-third-term}
  \end{align}
  where \eqref{eq:kronecker-proof-summation-by-parts} follows from summation by parts
  \ifverbose %
    , \eqref{eq:kronecker-proof-consolidating-after-adding-and-subtracting-S_m} follows from breaking the sum up into $i = 1,\dots, N-1$ and $i=N ,\dots, k-1$ and simplifying the telescoping sum,
  \fi %
   and \eqref{eq:kronecker-proof-first-term} follows from the triangle inequality.
  We will now bound the terms in \eqref{eq:kronecker-proof-first-term}, \eqref{eq:kronecker-proof-second-term}, and \eqref{eq:kronecker-proof-third-term} separately.
  \paragraph*{Bounding \eqref{eq:kronecker-proof-first-term} by $2\delta$} For any $m \geq N^\star$, we have
  \begin{align}
    &\sup_\Pin \Prob \left ( \sup_{k \geq m} \left \lvert  S_k -\frac{b_k-b_N}{b_k} S_m \right \rvert \geq  \eps/3 \right )\\
    \ifverbose %
   \leq\ &\sup_\Pin \Prob \left ( \sup_{k \geq m} \left \lvert  S_k - S_m \right \rvert \geq \eps / 6 \right ) + \sup_\Pin \Prob \left ( \sup_{k \geq m} \left \lvert  \frac{b_N}{b_k} S_m \right \rvert \geq \eps/6 \right )\\
    \fi %
    \leq\ &\underbrace{\sup_\Pin \Prob \left ( \sup_{k \geq m} \left \lvert  S_k - S_m \right \rvert \geq \eps / 6 \right )}_{< \delta} + \underbrace{\sup_\Pin \Prob \left (  |S_m| \geq \frac{\eps b_m}{6b_N} \right )}_{< \delta} < 2\delta.
  \end{align}
  where the last inequality follows from the conditions imposed on $N^\star \geq N$ in \eqref{eq:kronecker-proof-condition-cauchy} and \eqref{eq:kronecker-proof-condition-stochastic-boundedness} combined with the fact that $\eps b_m / 6 b_N \geq B$ for all $m \geq N^\star$ as in \eqref{eq:kronecker-proof-condition-increasing-b_m}.
  \paragraph*{Bounding \eqref{eq:kronecker-proof-second-term} by $\delta$} For any $m \geq N^\star$, we have
  \begin{align}
    &\sup_\Pin\Prob \left ( \sup_{k \geq m} \left \lvert \frac{1}{b_k}\sum_{i=1}^{N-1} (b_{i+1} - b_i) S_i \right \rvert \geq \eps / 3 \right )\\
    \ifverbose %
    =\ &\sup_\Pin\Prob \left ( \left \lvert\sum_{i=1}^{N-1} (b_{i+1} - b_i) S_i \right \rvert \geq \eps b_m/3 \right )\\
    \fi %
    \leq\ &\sup_\Pin\Prob \left ( \sum_{i=1}^{N-1} \left \lvert b_{i+1} - b_i\right \rvert \cdot |S_i|  \geq \eps b_m/3 \right )\\
    \leq\ &\sup_\Pin\Prob \left ( \sum_{i=1}^{N-1} \left \lvert S_i \right \rvert \geq \frac{\eps b_m}{3b_{N}} \right ) < \delta,
  \end{align}
  which follows from the condition imposed on $N^\star$ in \eqref{eq:kronecker-proof-condition-stochastic-boundedness-2}.
  \paragraph*{Bounding \eqref{eq:kronecker-proof-third-term} by $\delta$} For any $m \geq N$, we have
  \begin{align}
    &\sup_{\Pin}\Prob \left ( \sup_{k \geq m} \left \lvert \frac{1}{b_k} \sum_{i=N}^{k-1} (b_{i+1} - b_i)(S_i - S_m) \right \rvert \geq \eps / 3 \right )\\
    \ifverbose %
    \leq\ &\sup_{\Pin}\Prob \left ( \sup_{k \geq m} \frac{1}{b_k} \sum_{i=N}^{k-1} (b_{i+1} - b_i)|S_i - S_m| \geq \eps/3 \right ) \\
    \fi %
    \leq\ &\sup_{\Pin}\Prob \left ( \sup_{k \geq m} \frac{1}{b_k} \sum_{i=N}^{k-1} (b_{i+1} - b_i)\eps / 6 \geq \eps/3 \right ) + \underbrace{\sup_\Pin \Prob \left ( \sup_{k\geq N} |S_k - S_m| \geq \eps/6 \right )}_{ < \delta} \label{eq:kronecker-proof-third-term-using-cauchy}\\
    <\ &\underbrace{\sup_\Pin \1 \left \{ \sup_{k \geq m} \frac{b_k-b_N}{b_k}  \geq 2 \right \}}_{= 0} +\ \delta
  \end{align}
  which follows from the conditions imposed on $N$ in \eqref{eq:kronecker-proof-condition-cauchy} and the fact that $\supkm (b_k - b_N)/b_k \leq 1$.
  Putting the bounds in \eqref{eq:kronecker-proof-first-term}, \eqref{eq:kronecker-proof-second-term}, and \eqref{eq:kronecker-proof-third-term} together, we have that for any $m \geq N^\star$,
  \begin{equation}
    \sup_{P \in \Pcal} \Prob \left ( \sup_{k \geq m} \left \lvert  \frac{1}{b_k}\sum_{i=1}^k b_i Z_i  \right \rvert \geq \eps  \right ) < 4\delta,
  \end{equation}
  which yields the desired result, completing the proof.
\end{proof}

\subsection{Distribution-uniform Borel-Cantelli lemmas}\label{section:borel-cantelli}
In order to show that $\Pcal$-UI of certain finite absolute moments is in fact \emph{necessary} for the $\Pcal$-uniform SLLN to hold --- i.e.~the result of \cref{theorem:mz-slln}(iii) --- we rely on a $\Pcal$-uniform generalization of the \emph{second} Borel-Cantelli lemma. Before discussing the second Borel-Cantelli lemma, let us briefly discuss the first. A natural desideratum for a $\Pcal$-uniform first Borel-Cantelli lemma would be to say that for events $\infseqn{E_n}$ in $\Fcal$, if $\lim_{m}\sup_\Pin\sum_{k = m}^\infty \Prob(E_k) = 0$, then $\lim_{m} \sup_\Pin \Prob \left ( \bigcup_{k = m}^\infty E_k \right ) = 0$. Indeed, this is trivially satisfied since
  \begin{equation}
    \lim_\mto \sup_\Pin \Prob \left ( \bigcup_{k = m}^\infty E_k \right ) \leq \lim_\mto \sup_\Pin \sum_{k=m}^\infty \Prob \left ( E_k \right ).
  \end{equation}
  For this reason, we do not dwell on the first Borel-Cantelli lemma, but instead shift our attention to the second since its $\Pcal$-uniform generalization (and the proof thereof) is nontrivial in comparison and is central to the proof of \cref{theorem:mz-slln}(iii).

\begin{lemma}[The second $\Pcal$-uniform Borel-Cantelli lemma]\label{lemma:second-borel-cantelli}
  Let $\infseqn{E_n}$ be independent events such that
  \begin{equation}\label{eq:second-borel-cantelli-assumption}
    0 < \lim_{m \to \infty} \sup_\Pin \sum_{k=m}^\infty \Prob(E_k) \leq \infty.
  \end{equation}
  Then the tail probability of infinitely many of them occurring does not $\Pcal$-uniformly vanish, i.e.
 \begin{equation}\label{eq:second-borel-cantelli-conclusion}
   \lim_{m \to \infty} \sup_\Pin \Prob \left ( \bigcup_{k = m}^\infty E_k \right ) > 0.
 \end{equation} 
  
\end{lemma}

\begin{proof} The proof proceeds by a direct calculation. Writing out the limit in \eqref{eq:second-borel-cantelli-conclusion}, we have
  \begin{align}
   \lim_{m \to \infty} \sup_\Pin \Prob \left ( \bigcup_{k = m}^\infty E_k \right )
   =\ &\lim_{m \to \infty} \sup_\Pin \left \{  1-\Prob \left (  \bigcap_{k = m}^\infty E_k^c \right ) \right \}\\
    \ifverbose %
   =\ & 1- \lim_{m \to \infty} \inf_\Pin  \Prob \left ( \bigcap_{k = m}^\infty E_k^c \right )\\
    \fi %
    =\ &1-\lim_{m \to \infty} \inf_\Pin \lim_{t\to \infty}\Prob \left ( \bigcap_{k = m}^t E_k^c \right ) \label{eq:second-borel-cantelli-proof-nested}\\
    \ifverbose %
    =\ &1-\lim_{m \to \infty} \inf_\Pin \lim_{t\to \infty}\prod_{k=m}^t\Prob \left ( E_k^c \right )\\
    \fi %
    =\ &1-\lim_{m \to \infty} \inf_\Pin \lim_{t\to \infty}\prod_{k=m}^t(1-\Prob(E_k)) \label{eq:second-borel-cantelli-proof-independence}\\
    \ifverbose %
    \geq\ &1-\lim_{m \to \infty} \inf_\Pin \lim_{t\to \infty} \exp \left \{ -\sum_{k=m}^t\Prob(E_k) \right \}\\
    \fi %
    \geq\ &1-\exp \left \{ -\lim_{m \to \infty} \sup_\Pin \sum_{k=m}^\infty\Prob(E_k) \right \} > 0,\label{eq:second-borel-cantelli-proof-nonzero-lim}
  \end{align}
  where \eqref{eq:second-borel-cantelli-proof-nested} follows from the fact that the intersections $\infseqt{\bigcap_k^t E_k^c}$ are nested, \eqref{eq:second-borel-cantelli-proof-independence} exploits independence of $\infseqn{E_n}$, and \eqref{eq:second-borel-cantelli-proof-nonzero-lim} follows from the assumption in \eqref{eq:second-borel-cantelli-assumption}. This completes the proof.
\end{proof}

\section{Proof details for Theorems~\ref*{theorem:mz-slln} and~\ref*{theorem:slln-noniid}}\label{section:complete proofs}

With the introduction of the distribution-uniform analogues of Kolmogorov's three-series theorem (\cref{theorem:kolmogorov-three-series}), Kronecker's lemma (\cref{lemma:kronecker}), and the second Borel-Cantelli lemma (\cref{lemma:second-borel-cantelli}), we are ready to complete the proof details for our main results in Theorems~\ref{theorem:mz-slln} and~\ref{theorem:slln-noniid}.

\subsection{Proof details for \cref*{theorem:mz-slln}}\label{proof:mz-slln}

\begin{proof}[Proof details for \cref{theorem:mz-slln}(i)]
  Given the proof outline following \cref{theorem:mz-slln}, it only remains to show that the $\Pcal$-UI condition
\begin{equation}\label{eq:mz-slln-proof-uniform-integrability}
  \lim_\Lto \sup_\Pin \EE_\Prob \left ( |X - \EE_\Prob(X)|^q \cdot \1{\{|X - \EE_\Prob(X)|^q > L\}} \right ) = 0
\end{equation}
is sufficient to satisfy the conditions of both the $\Pcal$-uniform Kolmogorov three series theorem (\cref{theorem:kolmogorov-three-series}) and those of the $\Pcal$-uniform boundedness theorem (\cref{theorem:boundedness three series theorem}) for appropriately truncated and scaled versions of $\infseqn{X_n}$. These conditions are shown to hold in Lemmas~\ref{lemma:sufficient-conditions-puniform-kolmogorov-three-series-iid} and~\ref{lemma:satisfying uniform boundedness three series for q bigger than 1}, respectively.
\end{proof}

\begin{proof}[Proof details for \cref{theorem:mz-slln}(ii)]
  Similar to the proof of \cref{theorem:mz-slln}(i), it suffices to show that the \emph{uncentered} $\Pcal$-UI condition
  \begin{equation}\label{eq:mz-slln-proof-uniform-integrability-qlessthan1}
    \lim_\Lto \sup_\Pin \EE_\Prob \left ( |X|^q \1{\{ |X|^q > L \}} \right ) = 0
  \end{equation}
  is sufficient to satisfy the conditions of Theorems~\ref{theorem:kolmogorov-three-series} and~\ref{theorem:boundedness three series theorem} for appropriately truncated and scaled versions of $\infseqn{X_n}$, the details of which are provided in Lemmas~\ref{lemma:sufficient-conditions-m-z-qlessthan1} and~\ref{lemma:satisfying uniform boundedness three series for q smaller than 1}.
\end{proof}

\begin{proof}[Proof details for \cref{theorem:mz-slln}(iii)]
  Suppose that $\Pcal$ is a class of distributions for which the $\Pcal$-UI condition does not hold, i.e.
  \begin{equation}
    0 < \lim_\Lto \sup_\Pin \EE_\Prob \left ( |X - \mu(P;q)|^q \1{\{ |X-\mu(P;q)|^q > L \}} \right ) \leq \infty,
  \end{equation}
  recalling that $\mu(P;q) = \EE_\Prob(X)$ when $1 \leq q < 2$ and $\mu(P ; q) = 0$ when $ 0 < q < 1$. Then our goal is to show that
  \begin{equation}\label{eq:not-convergence}
    \lim_{m \to \infty}\sup_{P \in \Pcal} \Prob \left ( \supkm \left \lvert \frac{1}{k^{1/q}} \sum_{i=1}^k (X_i - \mu(P;q)) \right \rvert \geq \frac{1}{2} \right ) > 0.
\end{equation}
Indeed, pre-multiplying the above probability by 2 for any $P\in \Pcal$ and any $m \geq 1$, consider the partial sums $S_n := \sum_{i=1}^n (X_i - \mu(P; q))$ and note that
\begin{align}
  2\Prob \left ( \sup_{k \geq m-1} \frac{1}{k^{1/q}} |S_{k}| \geq 1/2 \right ) &\geq\Prob \left ( \supkm \frac{1}{k^{1/q}} |S_{k}| \geq 1/2 \right ) + \Prob \left ( \sup_{k \geq m-1} \frac{1}{k^{1/q}} |S_{k}| \geq 1/2 \right ) \\
  \ifverbose %
                                                                               &\geq \Prob \left ( \supkm \frac{1}{k^{1/q}} |S_{k}| \geq 1/2 \right ) + \Prob \left ( \sup_{k \geq m-1} \frac{1}{(k+1)^{1/q}} |S_{k}| \geq 1/2 \right )\\
                                                                               &= \Prob \left ( \supkm \frac{1}{k^{1/q}} |S_{k}| \geq 1/2 \right ) + \Prob \left ( \supkm \frac{1}{k^{1/q}} |S_{k-1}| \geq 1/2 \right )\\
  \fi %
  &\geq\Prob \left ( \supkm \frac{1}{k^{1/q}} (|S_k| + |S_{k-1}|) \geq 1 \right ) \\
  \ifverbose %
  \fi %
  &\geq \Prob \left ( \supkm \frac{1}{k^{1/q}}\left \lvert  X_k - \mu(P;q) \right \rvert \geq 1 \right )
\end{align}
  and hence by the $\Pcal$-uniform second Borel-Cantelli lemma (\cref{lemma:second-borel-cantelli}), it suffices to show that
  \begin{equation}\label{eq:mz-slln iii proof sufficient condition}
    \lim_{m \to \infty}\sup_{P \in \Pcal} \sum_{k=m}^\infty \Prob \left ( |X - \mu(P;q)| > k^{1/q} \right ) > 0,
  \end{equation}
  from which we will obtain \eqref{eq:not-convergence}. Indeed, by \citet[Theorem 2.1]{hu2017note} --- or as shown directly in \cref{lemma:uniformly-integrable-iff-uniformly-tail-vanishing} --- we have that for any random variable $Y$,
  \begin{align}
    \lim_\mto \sup_\Pin \sum_{k=m}^\infty \Prob \left ( |Y| > k \right ) = 0 \quad \text{if and only if} \quad\lim_\Lto \sup_\Pin \EE_\Prob \left ( |Y|\1{\left \{ |Y| > L \right \}}\right )  = 0,
  \end{align}
  and hence if the $q^\tth$ moment is not $\Pcal$-UI, we must have that \eqref{eq:mz-slln iii proof sufficient condition} holds. This completes the proof.
\end{proof}

\subsubsection{Sufficient conditions for the \texorpdfstring{$\Pcal$}{distribution}-uniform Kolmogorov three-series theorem}\label{section:sufficient-conditions-puniform-kolmogorov-three-series-iid}

In what follows, we verify that the tails of the three series of \cref{theorem:kolmogorov-three-series} vanish $\Pcal$-uniformly when $X$ has a $\Pcal$-UI $q^\tth$ moment, thereby enabling the application of the $\Pcal$-uniform Kolmogorov three-series theorem. Lemmas~\ref{lemma:sufficient-conditions-puniform-kolmogorov-three-series-iid} and~\ref{lemma:sufficient-conditions-m-z-qlessthan1} consider the cases of $1 < q < 2$ and $0 < q < 1$, respectively.

\begin{lemma}[Sufficient conditions in the identically distributed case when $1 < q < 2$]\label{lemma:sufficient-conditions-puniform-kolmogorov-three-series-iid}
  Let $\infseqn{X_n}$ be \iid{} random variables on the probability spaces $(\Omega, \Fcal, \Pcal)$ and let $Y_n := X_n - \EE_\Prob X$ be their centered versions for each $P \in \Pcal$. Suppose that the $q^\text{th}$ moment is UI for some $1 < q < 2$.
  Then, the three conditions of the $\Pcal$-uniform Kolmogorov three-series theorem are satisfied for $Z_n := Y_n/n^{1/q}$
  with $c = 1$.
\end{lemma}

\begin{proof}
  Throughout the proofs for the three series, consider the truncated random variable $Z_n^\leqo$ given by
  \begin{equation}
    Z_n^\leqo := Z_n \1\{|Z_n| \leq 1\}.
    \end{equation}
  Let us now separately show that the tails of the three series vanish $\Pcal$-uniformly.
  \paragraph*{The first series} Writing out the first series $\sup_\Pin \sum_{n=m}^\infty |\EE_\Prob Z_n^\leqo|$ for any $m \geq 1$ and noticing that $\EE_\Prob(Y\1 \{ |Y| \leq n^{1/q} \}) = -\EE_\Prob (Y\1 \{ |Y| > n^{1/q} \})$ since $\EE_\Prob(Y) = 0$, we have
  \begin{align}
    \sup_\Pin \sum_{n=m}^\infty  \left \lvert \EE_\Prob Z_n^\leqo \right \rvert &= \sup_\Pin \sum_{n=m}^\infty \left \lvert \EE_\Prob \left ( \frac{Y\1 \{ |Y| \leq n^{1/q} \}}{n^{1/q}} \right ) \right \rvert \\
&= \sup_\Pin \sum_{n=m}^\infty \left \lvert \EE_\Prob \left ( \frac{-Y\1 \{ |Y| > n^{1/q} \}}{n^{1/q}} \right ) \right \rvert\\
    \ifverbose %
    &= \sup_\Pin \sum_{n=m}^\infty \left \lvert \EE_\Prob \left ( \frac{-Y\1 \{ |Y| > n^{1/q} \}}{n^{1/q}} \right ) \right \rvert \\
    \fi %
    &\leq \sup_\Pin \sum_{n=m}^\infty \EE_\Prob \left ( \frac{|Y|\1 \{ |Y| > n^{1/q} \}}{n^{1/q}} \right )\\
                                                  &= \sup_\Pin \sum_{n=m}^\infty \sum_{k=n}^\infty \EE_\Prob \left ( \frac{|Y| \1 \{ k^{1/q} < |Y| \leq (k+1)^{1/q} \}}{n^{1/q}} \right )\\
    \ifverbose %
                                                  &= \sup_\Pin \sum_{n=1}^\infty \sum_{k=n}^\infty \1 \{ k \geq n \geq m \} \cdot \EE_\Prob \left ( \frac{|Y| \1 \{ k^{1/q} < |Y| \leq (k+1)^{1/q} \} }{n^{1/q}} \right )\\
                                                  &= \sup_\Pin \sum_{k=m}^\infty \sum_{n=m}^k \EE_\Prob \left ( \frac{|Y| \1 \{ k^{1/q} < |Y| \leq (k+1)^{1/q} \} }{n^{1/q}} \right )\\
    \fi %
                                                 &\leq \sup_\Pin \sum_{k=m}^\infty \EE_\Prob \left ( |Y| \1 \{ k^{1/q} < |Y| \leq (k+1)^{1/q} \} \right ) \cdot \sum_{n=1}^k \frac{1}{n^{1/q}}.
  \end{align}
  Now, there exists some constant $C_q > 0$ depending only on $q$ so that $\sum_{n=1}^k 1/n^{1/q} \leq C_q k / (k+1)^{1/q}$, thus
  \begin{align}
    \sup_\Pin \sum_{n=m}^\infty \left \lvert \EE_P Z_n^\leqo \right \rvert
    &\leq \sup_\Pin \sum_{k=m}^\infty \EE_P \left ( |Y| \1 \{ k^{1/q} < |Y| \leq (k+1)^{1/q} \} \right ) \cdot C_q \frac{k}{(k+1)^{1/q}} \\
    \ifverbose %
                                                  &\leq \sup_\Pin \sum_{k=m}^\infty \cancel{(k+1)^{1/q}} \Prob \left ( k^{1/q} < |Y| \leq (k+1)^{1/q} \right ) \cdot C_q \frac{k}{\cancel{(k+1)^{1/q}}} \\
    \fi %
    &\leq C_q \sup_\Pin \sum_{k=m}^\infty k \Prob \left ( k^{1/q} < |Y| \leq (k+1)^{1/q} \right ) \label{eq:part of first series used in second series proof}\\
    &= C_q \sup_\Pin \sum_{n=1}^\infty \Prob \left ( |Y|^q > (m \lor n) \right )\\
    \ifverbose %
    &= C_q \sup_\Pin \sum_{n=1}^\infty \Prob \left ( |Y|^q\1\{ |Y|^q > m  \} > n \right )\\
    \fi %
    &\leq C_q \sup_\Pin \EE_P \left ( |Y|^q \1\{|Y|^q > m\} \right ),
  \end{align}
  where the last inequality follows from the fact that $|Y|^q > (m\lor n)$ if and only if $|Y|^q \1\{ |Y|^q > m\lor n\} > (m \lor n)$ and using the expectation-tail-probability identity.
  \ifverbose %
    The above follows from the fact that
  \begin{align}
    &C_q \sup_\Pin \sum_{k=m}^\infty k \Prob \left ( k^{1/q} < |Y| \leq (k+1)^{1/q} \right ) \\
    =\ &C_q \sup_\Pin \sum_{k=1}^\infty \sum_{n=1}^k \1 \{ k \geq m \lor n \} \cdot \Prob \left ( k^{1/q} < |Y| \leq (k+1)^{1/q} \right ) \\
                                                  =\ &C_q \sup_\Pin \sum_{n=1}^\infty \sum_{k=m \lor n}^\infty \Prob \left ( k^{1/q} < |Y| \leq (k+1)^{1/q} \right ) \\
                                                                                                                                          =\ &C_q \sup_\Pin \sum_{n=1}^\infty \Prob \left ( |Y|^q > (m \lor n) \right )
  \end{align}
  and noticing that $|Y|^q > (m\lor n)$ if and only if $|Y|^q \1\{|Y|^q > m\lor n\} > (m \lor n)$ so that
  \begin{align}
                                                 C_q \sup_\Pin \sum_{n=1}^\infty \Prob \left ( |Y|^q > (m \lor n) \right ) &= C_q \sup_\Pin \sum_{n=1}^\infty \Prob \left ( |Y|^q\1\{ |Y|^q > m \lor n \} > (m \lor n) \right ) \\
                                                  &\leq C_q \sup_\Pin \sum_{n=1}^\infty \Prob \left ( |Y|^q\1\{ |Y|^q > m  \} > n \right ) \\
                                                  &\leq C_q \sup_\Pin \EE_P \left ( |Y|^q \1\{|Y|^q > m\} \right ),
  \end{align}
  and hence by the $\Pcal$-UI of the $q^\tth$ moment as in \eqref{eq:uniform-integrability-qth-moment}, we have that
  \begin{equation}
    \lim_\mto \sup_\Pin \sum_{n=m}^\infty \left \lvert \EE_P Z_n^\leqo \right \rvert = 0.
  \end{equation}
  \fi %
  Taking limits as $m \to \infty$ completes the argument for the first series.
  \paragraph*{The second series} Writing out the second series $\sup_\Pin \sum_{n=m}^\infty \Var_P (Z_n^\leqo)$ for any $m \geq 1$ and performing a direct calculation, we have
  \begin{align}
    \sup_\Pin \sum_{n=m}^\infty \Var_P Z_n^\leqo &\leq \sup_\Pin \sum_{n=m}^\infty \EE_P [ (Z_n^\leqo)^2] \\
    \ifverbose %
                                             &= \sup_\Pin \sum_{n=m}^\infty \EE_P \left [ \frac{Y^2}{n^{2/q}} \1 \{ |Y| \leq n^{1/q} \} \right ] \\
                                             &= \sup_\Pin \sum_{n=1}^\infty \1 \{n \geq m \} \cdot \EE_P \left [ \frac{Y^2}{n^{2/q}} \1\{|Y| \leq n^{1/q}\} \right ] \\
    \fi %
                                             &\leq \sup_\Pin \sum_{n=1}^\infty \1 \{n \geq m \} \cdot \EE_P \left [ \frac{Y^2}{n^{2/q}} \1\{|Y| \leq n^{1/q}\} \right ] \\
    \ifverbose %
                                             &= \sup_\Pin \sum_{n=1}^\infty \sum_{k=1}^n \1 \{n \geq m \} \cdot \EE_P \left [ \frac{Y^2}{n^{2/q}} \1\{(k-1)^{1/q} < |Y| \leq k^{1/q}\} \right ] \\
                                             &= \sup_\Pin \sum_{k=1}^\infty \sum_{n=k}^\infty \1 \{n \geq k \lor m \} \cdot \EE_P \left [ \frac{Y^2}{n^{2/q}} \1\{(k-1)^{1/q} < |Y| \leq k^{1/q}\} \right ] \\
    \fi %
                                                 &= \sup_\Pin \sum_{k=1}^\infty \EE_P \left [ Y^2 \1\{(k-1)^{1/q} < |Y| \leq k^{1/q}\} \right ] \sum_{n=k \lor m}^\infty \frac{1}{n^{2/q}}.
  \end{align}
  Now, there exists a constant $C_q > 0$ depending only on $q$ so that $\sum_{n=k\lor m}^\infty 1/n^{2/q} \leq C_q (k\lor m) / (k \lor m)^{2/q}$, and hence we have
  \begin{align}
    \sup_\Pin \sum_{n=m}^\infty \Var_P Z_n^\leqo
                                             &\leq C_q \Bigg \{ \underbrace{\sup_\Pin \sum_{k=1}^{m} \frac{m}{m^{2/q}} \cdot \EE_P \left [ Y^2 \1\{(k-1)^{1/q} < |Y| \leq k^{1/q}\} \right ]}_{(\star_\leq)} + \\
    &\quad\quad\quad \underbrace{\sup_\Pin \sum_{k=m}^{\infty} \frac{k}{k^{2/q}}\EE_P \left [ Y^2 \1\{(k-1)^{1/q} < |Y| \leq k^{1/q}\} \right ]}_{(\star_\geq)} \Bigg \},
  \end{align}
  and we will now separately show that $(\star_\leq) \to 0$ and $(\star_\geq ) \to 0$ as $m \to \infty$ using different arguments. Focusing first on $(\star_\leq)$, we invoke Lemma A.2 in the Supplementary Material \citep{supplement} with $p = 2$ and let $\phi(x) \equiv x g(x);\ x \geq 0$ be a function where $g$ has the following three key properties: (1) it is diverging to $\infty$, (2) it satisfies $a^{2/q} / b^{2/q} \leq a g(a) / (b g(b)) \equiv \phi(a) / \phi(b)$ whenever $0 \leq a \leq b$ and $b > 0$, and (3) it satisfies $\sup_\Pin \EE_P \left [|Y|^q g(|Y|^q) \right ] < \infty$.
  Writing out $(\star_\leq)$ and exploiting these three properties, we have
  \begin{align}
    (\star_\leq) %
    &= m \cdot \sup_\Pin \EE_P \left [ \frac{Y^2}{m^{2/q}} \1\{|Y| \leq m^{1/q}\} \right ]\\
    \ifverbose %
    &= m \cdot \sup_\Pin \EE_P \left [ \frac{(|Y|^q)^{2/q}}{m^{2/q}} \1\{|Y|^q \leq m\} \right ]\\
    \fi %
    &\leq m \cdot \sup_\Pin \EE_P \left [ \frac{\phi(|Y|^q)}{\phi(m)} \1\{|Y| \leq m^{1/q}\} \right ] \label{eq:second series proof, upper bound from phi}\\
    \ifverbose %
    &\leq \cancel{m} \cdot \sup_\Pin \EE_P \left [ \frac{\phi(|Y|^q)}{\cancel{m}\cdot g(m)} \right ]\\
    \fi %
                 &\leq \underbrace{\frac{1}{g(m)}}_{ \to 0} \underbrace{\sup_\Pin \EE_P \phi(|Y|^q)}_{< \infty},
  \end{align}
  where \eqref{eq:second series proof, upper bound from phi} follows from property (2) of $g$ described above, and the last inequality uses properties (1) and (3).
  Therefore, $\lim_\mto (\star_\leq) = 0$. Turning our focus to $(\star_\geq)$,
  \begin{align}
    (\star_\geq) %
    &\leq \sup_\Pin \sum_{k=m}^{\infty} k \Prob \left ((k-1)^{1/q} < |Y| \leq k^{1/q} \right ),
  \end{align}
  and the above vanishes as $m \to \infty$ by the proof for the first series beginning from \eqref{eq:part of first series used in second series proof}. Putting the analyses for $(\star_\leq)$ and $(\star_\geq)$ together and taking limits as $\mto$ completes the argument for the second series.
  \ifverbose %
    That is,
  \begin{equation}
    \lim_\mto \sup_\Pin \sum_{n=m}^\infty \Var_P Z_n^\leqo  =0.
  \end{equation}
  \fi %

  \paragraph*{The third series} Once again using the fact that $|Y|^q > k$ if and only if $|Y|^q \1\{ |Y|^q > k\} > k$ and writing out the third series for any $ m\geq 1$, we have
  \begin{align}
    \sup_\Pin \sum_{k=m}^\infty \Prob \left ( \left \lvert \frac{Y}{k^{1/q}} \right \rvert > 1 \right ) 
    &= \sup_\Pin \sum_{k=1}^\infty \1 \{ k \geq m \} \Prob(|Y|^q \1\{|Y|^q > k \} > k) \\
         &\leq \sup_\Pin \sum_{k=1}^\infty \1 \{ k \geq m \} \Prob(|Y|^q \1\{|Y|^q > m \} > k)\\
    \ifverbose %
         &\leq \sup_\Pin \sum_{k=1}^\infty \Prob(|Y|^q  > k)\\
    \fi %
    & \leq \sup_\Pin \EE_P\left (|Y|^q \1\{|Y|^q > m \} \right ),
  \end{align}
  which used the fact that $\Prob (|Y|^q \1\{ |Y|^q > k \} > k) \leq \Prob (|Y|^q \1\{ |Y|^q > m \} > k)$ whenever $k \geq m$ combined with the expectation-tail-probability identity. Taking limits as $\mto$ completes the argument for the third series, concluding the proof of Lemma~\ref{lemma:sufficient-conditions-puniform-kolmogorov-three-series-iid}.
  \ifverbose %
  Some details:
  \begin{align}
    &\sup_\Pin \sum_{k=1}^\infty \1 \{ k \geq m \} \Prob(|Y|^q \1\{|Y|^q > m \} > k)\\
    \leq\ &\sup_\Pin \sum_{k=1}^\infty \Prob(|Y|^q \1\{|Y|^q > m \} > k)\\
    \leq\ &\sup_\Pin \EE_P\left (|Y|^q \1\{|Y|^q > m \} \right ),
  \end{align}
  and thus using the $\Pcal$-UI $q^\text{th}$ moment of $Y$, we have that
  \begin{equation}
   \lim_{m \to \infty} \sup_\Pin \sum_{k=m}^\infty \Prob \left ( |Y| > k^{1/q} \right ) = 0,
  \end{equation}
  completing the proof of the third series, and hence the entire lemma.
  \fi %
\end{proof}

\begin{lemma}[Sufficient conditions for the three series with $0 < q < 1$]\label{lemma:sufficient-conditions-m-z-qlessthan1}
  Given the same setup as \cref{lemma:sufficient-conditions-puniform-kolmogorov-three-series-iid}, suppose that $X$ has a $\Pcal$-UI (uncentered) $q^\tth$ moment for some $0 < q < 1$:
  \begin{equation}
    \lim_{L \to \infty}\sup_\Pin \EE_P \left ( |X|^{q} \1\{ |X|^q > L \} \right ) = 0.
  \end{equation}
  Then, the three conditions of the $\Pcal$-uniform Kolmogorov three-series theorem are satisfied for $Z_n := X_n / n^{1/q}$
  with $c = 1$.
\end{lemma}
\begin{proof}
  Once again, consider the truncated random variable $Z_n^\leqo$ given by
  \begin{equation}
    Z_n^\leqo := Z_n \1\{|Z_n| \leq 1\}.
    \end{equation}
  Let us now separately show that the tails of the three series vanish $\Pcal$-uniformly.
    \paragraph*{The first series}
  Writing out the first series for any $m \geq 1$ and using a similar argument to that of the second series in Lemma~\ref{lemma:sufficient-conditions-puniform-kolmogorov-three-series-iid}, we have that
  \begin{align}
    \sup_\Pin \sum_{n=m}^\infty  \left \lvert \EE_P Z_n^\leqo \right \rvert &\leq \sup_\Pin \sum_{n=m}^\infty \EE_P \left ( \frac{|X|\1 \{ |X| \leq n^{1/q} \}}{n^{1/q}} \right )\\
    \ifverbose %
    &\leq \sup_\Pin \sum_{n=1}^\infty \sum_{k=1}^n \1\{n \geq k \lor m\} \EE_P \left ( \frac{|X|\1 \{(k-1)^{1/q} < |X| \leq k^{1/q} \}}{n^{1/q}} \right )\\
    &= \sup_\Pin \sum_{k=1}^\infty \sum_{n=k}^\infty \1\{n \geq k \lor m\}\EE_P \left ( \frac{|X|\1 \{(k-1)^{1/q} < |X| \leq k^{1/q} \}}{n^{1/q}} \right )\\
    \fi %
                                                 &\leq \sup_\Pin \sum_{k=1}^\infty \EE_P \left ( |X|\1 \{(k-1)^{1/q} < |X| \leq k^{1/q} \} \right )\cdot \sum_{n=k \lor m}^\infty \frac{1}{n^{1/q}},
  \end{align}
  and thus there exists $C_q > 0$ depending only on $q \in (0, 1)$ so that
  \begin{align}
     \sup_\Pin \sum_{n=m}^\infty  \EE_P Z_n^\leqo %
                                                  &\leq C_q \underbrace{\sup_\Pin \sum_{k=1}^m \frac{m}{m^{1/q}} \cdot \EE_P \left ( |X| \1 \left \{ (k-1)^{1/q} < |X| \leq k^{1/q} \right \} \right )}_{(\star_\leq)}  +\\
                                                  &\quad C_q \underbrace{\sup_\Pin \sum_{k=m}^\infty \frac{k}{k^{1/q}}\cdot \EE_P \left ( |X| \1 \left \{ (k-1)^{1/q} < |X| \leq k^{1/q} \right \} \right )}_{(\star_\geq)},
  \end{align}
  and thus similarly to the proof for the second series in \cref{lemma:sufficient-conditions-puniform-kolmogorov-three-series-iid}, it suffices to show that $(\star_\leq) \to 0$ and $(\star_\geq) \to 0$ as $m \to \infty$. Focusing on $(\star_\leq)$ first, we have
  \begin{align}
    (\star_\leq) %
    &\leq \sup_\Pin m \cdot \EE_P \left (\frac{|X|}{m^{1/q}}\1 \left \{ |X| \leq m^{1/q} \right \} \right ).
  \end{align}
  By Lemma A.2 in the Supplementary Material \citep{supplement} applied with $p = 1$, there exists a function $\phi(x) = x g(x);\ x\geq 0$ where $g > 0$ is a function that (1) is diverging to $\infty$, (2) satisfies $a^{1/q} / b^{1/q} \leq a g(a) / (bg(b)) \equiv \phi(a) / \phi(b)$ whenever $0 \leq a \leq b$ and $b > 0$, and (3) satisfies
    $\sup_\Pin \EE_P \phi(|X|^q) < \infty$. Writing out $(\star_\leq)$ and exploiting properties (1)--(3), we have
  \begin{align}
    (\star_\leq) &\leq \sup_\Pin m \cdot \EE_P \left (\frac{|X|}{m^{1/q}}\1 \left \{ |X| \leq m^{1/q} \right \} \right )\\
                 &\leq \sup_\Pin m \cdot \EE_P \left (\frac{\phi(|X|^q)}{\phi(m)} \right )\\
    &=  \frac{\cancel{m}}{\cancel{m} g(m)} \underbrace{\sup_\Pin\EE_P \phi(|X|^q)}_{< \infty},
  \end{align}
  and since $\lim_\mto g(m) = \infty$, we have that $\lim_\mto (\star_\leq) = 0$.
  Moving to $(\star_\geq)$, we have that
  \begin{align}
    (\star_\geq) &= \sup_\Pin \sum_{k=m}^\infty \frac{k}{k^{1/q}}\cdot \EE_P \left ( |X| \1 \left \{ (k-1)^{1/q} < |X| \leq k^{1/q} \right \} \right )   \\
    \ifverbose %
                 &\leq \sup_\Pin \sum_{k=m}^\infty \frac{k}{\cancel{k^{1/q}}}\cdot \EE_P \left (\cancel{k^{1/q}} \1 \left \{ (k-1)^{1/q} < |X| \leq k^{1/q} \right \} \right ) \\
    \fi %
                 &\leq \sup_\Pin \sum_{k=m}^\infty k\cdot \Prob \left ( (k-1)^{1/q} < |X| \leq k^{1/q} \right ) \\
    &\leq \sup_\Pin \sum_{k=m}^\infty \EE_P \left [ (|X|^q + 1) \1 \{ k-1 < |X|^q \leq k \} \right ] \\
    & \leq \sup_\Pin \EE_P \left ( |X|^q \1\{ |X|^q > m - 1\} \right ) + \sup_\Pin \frac{\EE_P |X|^q }{m-1},
  \end{align}
  and thus $\lim_\mto (\star_\geq) = 0$. Putting $(\star_\leq)$ and $(\star_\geq)$ together and taking limits as $\mto$ completes the argument for the first series.

  \paragraph*{The second series} The second series proceeds similarly to that of \cref{lemma:sufficient-conditions-puniform-kolmogorov-three-series-iid}. Indeed, using the same arguments therein, notice that there exists a constant $C_q > 0$ depending only on $q \in (0, 1)$ so that
  \begin{align}
    \sup_\Pin \sum_{n=m}^\infty \Var_P Z_n^\leqo
                                             &\leq C_q \Bigg \{ \underbrace{\sup_\Pin \sum_{k=1}^{m} \frac{m}{m^{2/q}} \cdot \EE_P \left [ X^2 \1\{(k-1)^{1/q} < |X| \leq k^{1/q}\} \right ]}_{(\dagger_\leq)} + \\
    &\quad\quad\quad \underbrace{\sup_\Pin \sum_{k=m}^{\infty} \frac{k}{k^{2/q}}\EE_P \left [ X^2 \1\{(k-1)^{1/q} < |X| \leq k^{1/q}\} \right ]}_{(\dagger_\geq)} \Bigg \},
  \end{align}
  Focusing first on $(\dagger_\leq)$, since $(a/b)^2 \leq a/b$ whenever $0 \leq a \leq b$ and $b > 0$, we have that
  \begin{align}
    (\dagger_\leq) 
    &= m \cdot \sup_\Pin \EE_P \left [ \frac{X^2}{m^{2/q}} \1\{|X| \leq m^{1/q}\} \right ]\\
    &\leq m \cdot \sup_\Pin \EE_P \left [ \frac{|X|}{m^{1/q}} \1\{|X| \leq m^{1/q}\} \right ],
  \end{align}
  and now by the same argument used to show that $(\star_\leq)\to 0$ in the first series above, we have that $(\dagger_\leq) \to 0$ as $m \to \infty$. Turning to $(\dagger_\geq)$, we have that 
  \begin{align}
    (\dagger_\geq) &\leq \sup_\Pin \sum_{k=m}^\infty k \Prob \left ( (k-1)^{1/q} < |X| \leq k^{1/q} \right ),
  \end{align}
  and hence using the same argument as was used to bound $(\star_\geq)$ in the first series, we have that $(\dagger_\geq) \to 0$ as $m \to \infty$. Putting the limits for $(\dagger_\leq)$ and $(\dagger_\geq)$ together completes the argument for the second series.
  \paragraph*{The third series} The proof for the third series proceeds identically to that of \cref{lemma:sufficient-conditions-puniform-kolmogorov-three-series-iid} when $1 < q < 2$. This completes the proof of \cref{lemma:sufficient-conditions-m-z-qlessthan1}.
\end{proof}

\begin{lemma}[Satisfying the three series for stochastic boundedness when $1 < q < 2$]\label{lemma:satisfying uniform boundedness three series for q bigger than 1}
  Suppose that $\infseqn{X_n}$ are \iid{} and have a $\Pcal$-UI $q^\tth$ moment for $1 < q < 2$. Then the three series for uniform stochastic boundedness in \cref{theorem:boundedness three series theorem} are satisfied for the random variable $Z_n := (X_n - \EE_P(X)) / n^{1/q}$.
\end{lemma}

\begin{proof}
  We will handle the three series separately below. Throughout, let $Y_n := X_n - \EE_P(X)$ and $Z_{n,B}^\leqo := (Y/B) \1 \{|(Y/B)| \leq n^{1/q}\} / n^{1/q}$.
  \paragraph*{The first series}

  Writing out the first series $\sup_\Pin \sum_{n=1}^\infty |\EE_P ( Z_{n,B}^\leqo )|$ and performing calculations analogous to those for the first series in Lemma~\ref{lemma:sufficient-conditions-puniform-kolmogorov-three-series-iid}, we have
  \begin{align}
    \sup_\Pin \sum_{n=1}^\infty  |\EE_P Z_{n,B}^\leqo| &= \sup_\Pin \sum_{n=1}^\infty \left \lvert \EE_P \left ( \frac{(Y/B)\1\{|(Y/B)| \leq n^{1/q}\}}{n^{1/q}} \right ) \right \rvert\\
    \ifverbose %
    &= \sup_\Pin \sum_{n=1}^\infty \left \lvert \EE_P \left ( \frac{-(Y/B)\1 \{ |Y/B| > n^{1/q} \}}{n^{1/q}} \right ) \right \rvert\\
    \fi %
    \ifverbose %
    &\leq \sup_\Pin \sum_{n=1}^\infty \EE_P \left ( \frac{|Y/B|\1 \{ |Y/B| > n^{1/q} \}}{n^{1/q}} \right )\\
                                                  &=\sup_\Pin \sum_{n=1}^\infty \sum_{k=n}^\infty \EE_P \left ( \frac{|Y/B| \1 \{ k^{1/q} < |Y/B| \leq (k+1)^{1/q} \}}{n^{1/q}} \right )\\
                                                  &= \sup_\Pin \sum_{k=1}^\infty \sum_{n=1}^k \EE_P \left ( \frac{|Y/B| \1 \{k^{1/q} < |Y/B| \leq (k+1)^{1/q} \}}{n^{1/q}} \right )\\
    \fi %
                                                 &\leq \sup_\Pin \sum_{k=1}^\infty \EE_P \left ( |Y/B| \1 \{ k^{1/q} < |Y/B| \leq (k+1)^{1/q} \} \right ) \cdot \sum_{n=1}^k \frac{1}{n^{1/q}}.
  \end{align}
  Now, there exists some constant $C_q > 0$ depending only on $q$ so that $\sum_{n=1}^k 1/n^{1/q} \leq C_q k / (k+1)^{1/q}$ and thus
  \begin{align}
    \sup_\Pin \sum_{n=1}^\infty | \EE_P Z_{n,B}^\leqo |
    &\leq \sup_\Pin \sum_{k=1}^\infty \EE_P \left ( |Y/B| \1 \{ k^{1/q} < |Y/B| \leq (k+1)^{1/q} \} \right ) \cdot C_q \frac{k}{(k+1)^{1/q}} \\
    \ifverbose %
                                                  &\leq \sup_\Pin \sum_{k=1}^\infty \cancel{(k+1)^{1/q}} \Prob \left ( k^{1/q} < |Y/B| \leq (k+1)^{1/q} \right ) \cdot C_q \frac{k}{\cancel{(k+1)^{1/q}}} \\
    \fi %
                                                                                               &\leq C_q \cdot \sup_\Pin \sum_{k=1}^\infty k \Prob \left ( k^{1/q} < |Y/B| \leq (k+1)^{1/q} \right )\\
    \ifverbose %
     &= C_q \cdot \sup_\Pin \sum_{k=1}^\infty \sum_{n=1}^k \Prob \left ( k^{1/q} < |Y/B| \leq (k+1)^{1/q} \right ) \\
                                                   &= C_q \cdot \sup_\Pin \sum_{n=1}^\infty \sum_{k=n}^\infty \Prob \left ( k^{1/q} < |Y/B| \leq (k+1)^{1/q} \right ) \\
    \fi %
                            &= C_q \cdot \sup_\Pin \sum_{n=1}^\infty \Prob \left ( |Y / B|^q > n \right )\\
    \ifverbose %
                            &\leq C_q \cdot \sup_\Pin \EE_P \left ( |Y/B|^q \right )\\
    \fi %
                            &\leq C_q \cdot B^{-q} \underbrace{\sup_\Pin \EE_P \left ( |Y|^q \right )}_{< \infty}.
  \end{align}
  Taking limits as $B \to \infty$, we have the desired result:
  \begin{equation}
    \lim_{B \to \infty} \sup_\Pin \sum_{n=1}^\infty |\EE_P Z_{n,B}^\leqo| = 0,
  \end{equation}
  completing the proof for the first series for $\Pcal$-uniform stochastic boundedness.

  \paragraph*{The second series}
  Writing out the second series $\sup_\Pin \sum_{n=1}^\infty \Var_P (Z_{n,B}^\leqo)$ for any $B >0 $, we have
  \begin{align}
    \sup_\Pin \sum_{n=1}^\infty \Var_P Z_{n,B}^\leqo &=  \sup_\Pin \sum_{n=1}^\infty \left \{ \EE_P [ (Z_{n,B}^\leqo)^2] - [\EE_P Z_{n,B}^\leqo]^2 \right \} \\
    \ifverbose %
                                                     &\leq \sup_\Pin \sum_{n=1}^\infty \EE_P [ (Z_{n,B}^\leqo)^2]\\
    \fi %
                                             &\leq \sup_\Pin \sum_{n=1}^\infty \EE_P \left [ \frac{(Y/B)^2}{n^{2/q}} \1\{|Y/B| \leq n^{1/q}\} \right ] \\
    \ifverbose %
                                             &= \sup_\Pin \sum_{n=1}^\infty \sum_{k=1}^n \EE_P \left [ \frac{(Y/B)^2}{n^{2/q}} \1\{(k-1)^{1/q} < |Y/B| \leq k^{1/q}\} \right ] \\
                                             &= \sup_\Pin \sum_{k=1}^\infty \sum_{n=k }^\infty  \EE_P \left [ \frac{(Y/B)^2}{n^{2/q}} \1\{(k-1)^{1/q} < |Y/B| \leq k^{1/q}\} \right ] \\
    \fi %
                                                 &= \sup_\Pin \sum_{k=1}^\infty \EE_P \left [ (Y/B)^2 \1\{(k-1)^{1/q} < |Y/B| \leq k^{1/q}\} \right ] \sum_{n=k}^\infty \frac{1}{n^{2/q}},
  \end{align}
  and since there exists a constant $C_q > 0$ depending only on $q$ so that $\sum_{n=k}^\infty 1/n^{2/q} \leq C_q k / k^{2/q}$, we have
  \begin{align}
    \sup_\Pin \sum_{n=1}^\infty \Var_P Z_{n,B}^\leqo &\leq C_q \sup_\Pin \sum_{k=1}^\infty \frac{k}{k^{2/q}} \cdot \EE_P \left [ (Y/B)^2 \1\{(k-1)^{1/q} < |Y/B| \leq k^{1/q}\} \right ].
  \end{align}
  Separating the first term from the rest of the sum, we can write the above as
  \begin{align}
    &C_q \sup_\Pin \sum_{k=1}^\infty \frac{k}{k^{2/q}} \cdot \EE_P \left [ (Y/B)^2 \1\{(k-1)^{1/q} < |Y/B| \leq k^{1/q}\} \right ]\\
    \leq\ &\underbrace{C_q\sup_\Pin \EE_P \left [ (Y/B)^2 \1\{|Y/B| \leq 1\} \right ]}_{(\star)} + \\
    &\underbrace{C_q \sup_\Pin \sum_{k=2}^\infty \frac{k}{k^{2/q}} \cdot \EE_P \left [ (Y/B)^2 \1\{(k-1)^{1/q} < |Y/B| \leq k^{1/q}\} \right ]}_{(\dagger)}.
  \end{align}
  Notice that $(Y/B)^2 \1\{ |Y/B| \leq 1 \} \leq |Y/B|^q \1\{ |Y/B| \leq 1 \}$ with $P$-probability one for every $P \in \Pcal$ since $0 < q < 2$. Therefore, the first term $(\star)$ of the aforementioned sum can be upper-bounded as
  \begin{equation}
    (\star) = C_q \sup_\Pin \EE_P \left [ (Y/B)^2 \1\{|Y/B| \leq 1\} \right ] \leq \frac{C_q}{B^q} \sup_\Pin \EE_P |Y|^q.
  \end{equation}
  Turning now to the second term $(\dagger)$, we have
  \begin{align}
    (\dagger) &= C_q \sup_\Pin \sum_{k=2}^\infty \frac{k}{k^{2/q}} \cdot \EE_P \left [ (Y/B)^2 \1\{(k-1)^{1/q} < |Y/B| \leq k^{1/q}\} \right ]\\
    &\leq C_q \sup_\Pin \sum_{k=2}^{\infty} \frac{k}{\cancel{k^{2/q}}}\EE_P \left [ \cancel{k^{2/q}} \1\{(k-1)^{1/q} < |Y/B| \leq k^{1/q}\} \right ] \\
    \ifverbose %
                                                 &= C_q \sup_\Pin \sum_{k=2}^{\infty} k\Prob \left ( (k-1)^{1/q} < |Y/B| \leq k^{1/q} \right )\\
                                                 &= C_q \sup_\Pin \sum_{k=2}^{\infty} \sum_{n=1}^k\Prob \left ( (k-1)^{1/q} < |Y/B| \leq k^{1/q} \right )\\
    \fi %
                                                 &= C_q \sup_\Pin \sum_{n=1}^{\infty} \sum_{k=n}^\infty \1\{k \geq 2\} \Prob \left ( (k-1)^{1/q} < |Y/B| \leq k^{1/q} \right )\\
              &= C_q \sup_\Pin \Bigg [ \sum_{k=2}^\infty \Prob \left ( (k-1)^{1/q} < |Y/B| \leq k^{1/q} \right ) +
    \sum_{n=2}^\infty \sum_{k=n}^\infty \Prob \left ( (k-1)^{1/q} < |Y/B| \leq k^{1/q} \right )\Bigg ]\\
    \ifverbose %
              &= C_q \sup_\Pin \Bigg [ 2\sum_{k=2}^\infty \Prob \left ( (k-1)^{1/q} < |Y/B| \leq k^{1/q} \right ) +\\
    &\quad\quad\quad\quad\quad \sum_{n=3}^\infty \sum_{k=n}^\infty \Prob \left ( (k-1)^{1/q} < |Y/B| \leq k^{1/q} \right )\Bigg ]\\
    \fi %
              &\leq C_q \sup_\Pin \Bigg [ 2 \sum_{n=2}^\infty \sum_{k=n}^\infty \Prob \left ( (k-1)^{1/q} < |Y/B| \leq k^{1/q} \right )\Bigg ]\\
              &\leq 2C_q \sup_\Pin \sum_{n=2}^\infty \Prob \left (|Y/B| >  (n-1)^{1/q} \right )\\
    \ifverbose %
                                                 &= 2C_q  \sup_\Pin \sum_{n=1}^{\infty} \Prob \left ( |Y/B|^q > n \right )\\
                                                 &\leq 2 C_q \sup_\Pin \EE_P \left ( |Y/B|^q \right )\\
    \fi %
                                                 &\leq \frac{2C_q }{B^q} \sup_\Pin \EE_P \left ( |Y|^q \right ).
  \end{align}
  Putting the bounds on $(\star)$ and $(\dagger)$ together, we have
  \begin{align}
    \sup_\Pin \sum_{n=1}^\infty \Var_P Z_{n,B}^\leqo &\leq (\star) + (\dagger) \leq \frac{3 C_q}{B^q} \underbrace{\sup_\Pin \EE_P |Y|^q}_{< \infty}, 
  \end{align}
  so that when we take limits as $B \to \infty$, we obtain the desired result,
  \begin{equation}
    \lim_{B \to \infty} \supP \sum_{n=1}^\infty \Var_P Z_{n,B}^\leqo = 0,
  \end{equation}
  which completes the proof for the second series.
  
  \paragraph*{The third series}
  Writing out the series $\sup_\Pin \sum_{k=1}^\infty \Prob \left ( |Y / k^{1/q}| > B \right )$ for any $B > 0$ and using the expectation-tail sum identity, we have
  \begin{align}
    \sup_\Pin \sum_{k=1}^\infty \Prob \left ( |Y / k^{1/q}| > B \right )
    = \sup_\Pin \sum_{k=1}^\infty \Prob \left ( |Y / B|^q > k \right )
    \leq \frac{1}{B^q} \underbrace{\sup_\Pin \EE_P \left ( |Y|^q \right )}_{< \infty},
  \end{align}
  and we note that $\sup_\Pin \EE_P \left ( |Y|^q \right ) < \infty$ above since uniform integrability implies uniform boundedness. Consequently, we have that
  \begin{equation}
    \lim_{B \to \infty} \supP \sum_{k=1}^\infty \Prob \left ( |Y / k^{1/q}| > B \right ) = 0,
  \end{equation}
  completing the proof for the third series and hence the entire lemma.
\end{proof}

\begin{lemma}[Satisfying the three series for stochastic boundedness when $0 < q < 1$]\label{lemma:satisfying uniform boundedness three series for q smaller than 1}
  Suppose that $X$ has a $\Pcal$-UI (uncentered) $q^\tth$ moment for $0 < q < 1$. Then the three series for uniform stochastic boundedness in \cref{theorem:boundedness three series theorem} are satisfied for the random variable $Z_n := X_n/ n^{1/q}$.
\end{lemma}

\begin{proof}
  Similar to satisfying the conditions of the $\Pcal$-uniform Kolmogorov three-series theorem as in \cref{lemma:sufficient-conditions-m-z-qlessthan1}, satisfying the conditions of the $\Pcal$-uniform stochastic boundedness three-series theorem proceeds identically for the second and third series, and thus we focus solely on the first series here.
  Throughout, let $Z_{n,B}^\leqo := (X/B) \1 \{(X/B)| \leq n^{1/q}\} / n^{1/q}$.
  Writing out $\sup_\Pin \sum_{n=1}^\infty | \EE_P (Z_{n,B}^\leqo) |$ for any $B > 0$,
  \begin{align}
    &\sup_\Pin \sum_{n=1}^\infty \left \lvert \EE_P Z_{n,B}^\leqo \right \rvert\\
    \leq\ &\sup_\Pin \sum_{n=1}^\infty \EE_P \left ( \frac{|X/B|\1\{ |X/B| \leq n^{1/q} \}}{n^{1/q}} \right )\\
    \ifverbose %
    =\ &\sup_\Pin \sum_{n=1}^\infty \sum_{k=1}^n \EE_P \left ( \frac{|X/B|\1 \{(k-1)^{1/q} < |X/B| \leq k^{1/q}\}}{n^{1/q}} \right )\\
    \fi %
    =\ &\sup_\Pin \sum_{k=1}^\infty  \EE_P \left ( |X/B|\1 \{(k-1)^{1/q} < |X/B| \leq k^{1/q}\} \right )\cdot \sum_{n=k}^\infty \frac{1}{n^{1/q}}.
  \end{align}
  Now, following techniques from earlier proofs, notice that since $0 < q < 1$, there exists a constant $C_q$ depending only on $q$ so that $ \sum_{n=k}^\infty 1/n^{1/q} \leq C_q k / k^{1/q}$. Breaking up the sum similarly to how we did for the second series of uniform stochastic boundedness in \cref{lemma:satisfying uniform boundedness three series for q bigger than 1}, we have
  \begin{align}
    \sup_\Pin \sum_{n=1}^\infty \left \lvert \EE_P Z_{n,B}^\leqo \right \rvert &\leq \sup_\Pin \sum_{k=1}^\infty \EE_P \left ( |X/B|\1 \{(k-1)^{1/q} < |X/B| \leq k^{1/q} \} \right ) \cdot C_q \frac{k}{k^{1/q}}\\
                                                                               &\leq \underbrace{C_q \sup_\Pin \EE_P \left ( |X/B| \1 \{ |X/B| \leq 1 \right )}_{(\star)} +\\
    &\quad \underbrace{C_q \sup_\Pin \sum_{k=2}^\infty \EE_P \left ( |X/B|\1 \{(k-1)^{1/q} < |X/B| \leq k^{1/q} \} \right ) \frac{k}{k^{1/q}}}_{(\dagger)}.
  \end{align}
  First looking at the first term $(\star)$, we notice that $|X/B| \1\{ |X/B| \leq 1 \} \leq |X/B|^q \1\{ |X/B| \leq 1\}$ with $P$-probability one for every $P \in \Pcal$ since $0 < q < 1$, and thus
  \begin{equation}
    (\star) \leq \frac{C_q}{B^q} \sup_\Pin \EE_P |X|^q.
  \end{equation}
  Turning to the second term $(\dagger)$, we have that
  \begin{align}
    (\dagger) %
              &\leq C_q \sup_\Pin \sum_{k=2}^\infty \EE_P \left ( \cancel{k^{1/q}} \1 \{(k-1)^{1/q} < |X/B| \leq k^{1/q} \} \right ) \frac{k}{\cancel{k^{1/q}}}\\
    \ifverbose %
              &= C_q \sup_\Pin \sum_{k=1}^\infty \sum_{n=1}^k\1\{ k \geq 2\} \Prob \left ( (k-1)^{1/q} < |X/B| \leq k^{1/q} \right ) \\
              &= C_q \sup_\Pin \sum_{n=1}^\infty \sum_{k=n}^\infty \1\{ k \geq 2\} \Prob \left ( (k-1)^{1/q} < |X/B| \leq k^{1/q} \right ) \\
              &\leq 2 C_q \sup_\Pin \sum_{n=1}^\infty \Prob \left ( |X/B| > n^{1/q} \right ) \\
    \fi %
              &\leq \frac{2 C_q}{B^q} \sup_\Pin \EE_P |X|^q.
  \end{align}
  Combining the upper bounds on $(\star)$ and $(\dagger)$ and taking limits as $B \to \infty$, we have
  \begin{equation}
    \lim_{B \to \infty}\sup_\Pin \sum_{n=1}^\infty \left \lvert \EE_P Z_{n, B}^\leqo \right \rvert = 0,
  \end{equation}
  completing the proof of \cref{lemma:satisfying uniform boundedness three series for q smaller than 1}.
\end{proof}

\subsubsection{An equivalent criterion of uniform integrability}

\begin{lemma}[Uniform integrability is equivalent to uniformly vanishing sums of tail probabilities]\label{lemma:uniformly-integrable-iff-uniformly-tail-vanishing}
  Let $Y$ be a random variable on the probability spaces $(\Omega,\Fcal, \Pcal)$ and fix $q > 0$. Then $Y$ has a $\Pcal$-UI $q^\tth$ moment if and only if the tail sum of its tail probability is $\Pcal$-uniformly vanishing, meaning
  \begin{equation}\label{eq:uniformly-integrable-qth-moment-tail-lemma}
    \lim_{L \to \infty}\sup_\Pin \EE_P \left ( |Y|^q \1 \{ |Y|^q > L \} \right ) = 0\quad\text{if and only if}\quad\lim_{m \to \infty}\sup_\Pin \sum_{k=m}^\infty \Prob \left ( |Y|^q > k \right ) = 0.
  \end{equation}
\end{lemma}

The above lemma gives an equivalent criterion of $\Pcal$-UI and was shown in \citet[Theorem 2.1]{hu2017note} where they refer to the property in the right-hand side of \eqref{eq:uniform-integrability-qth-moment} as ``$W^\star$ uniform integrability'' \citep[Definition 1.6]{hu2017note}. Since \citet[Theorem 2.1]{hu2017note} is written in the context of uniform integrability for a family of random variables $\infseqn{X_n}$ defined on a single probability space (as contrasted with $\Pcal$-UI in \cref{section:introduction}), we provide a self-contained proof for completeness.
\begin{proof}
  The forward implication is not hard to show since for any $P \in \Pcal$,
  \begin{align}
    \sum_{k=m}^\infty \Prob(|Y|^q > k) &= \sum_{k=m}^\infty \Prob(|Y|^q\1\{ |Y|^q > m\} > k)\\
    \ifverbose %
                                                   &\leq \sum_{k=1}^\infty \Prob(|Y|^q \1 \{ |Y|^q > m \} > k)\\
    \fi %
    &\leq \int_{0}^\infty \Prob(|Y|^q \1 \{ |Y|^q > m \} > k) dk\\
    &= \EE_P(|Y|^q \1 \{ |Y|^q > m \}).
  \end{align}
  The reverse implication is more involved. Suppose that there exists a collection of distributions $\Pcal$ so that right-hand side of \eqref{eq:uniformly-integrable-qth-moment-tail-lemma} holds but the left-hand side does not.
  First, note that we can write the supremum over expectations in \eqref{eq:uniformly-integrable-qth-moment-tail-lemma} as
  \begin{align}
    \sup_\Pin \EE_P \left ( |Y|^q \1 \{|Y|^q > m \} \right ) &= \supP \int_{0}^\infty \Prob \left ( |Y|^q \1\{ |Y|^q > m \} > k \right ) dk\\
                                                  &\leq \sup_\Pin \sum_{k=0}^\infty \Prob \left ( |Y|^q \1 (|Y|^q > m) > k \right )\\
    &\leq \underbrace{\sup_\Pin \sum_{k=m}^\infty \Prob \left ( |Y|^q \1\{ |Y|^q > m\} > k \right )}_{(\star)} +\\
    &\quad\underbrace{\sup_\Pin \sum_{k=0}^{m-1} \Prob \left ( |Y|^q \1\{ |Y|^q > m\} > k \right )}_{(\dagger)},
  \end{align}
  and since $(\star) \to 0$ as $m \to \infty$, we must have that $(\dagger) \not \to 0$, and hence
  \begin{equation}\label{eq:limsup-of-lower-partial-sum}
    \limsup_{m \to \infty} \sup_\Pin \sum_{k=0}^{m-1} \Prob \left ( |Y|^q \1\{ |Y|^q > m\} > k \right ) > \eps
  \end{equation}
  for some $\eps > 0$, or in other words, no matter how large we take $M$ to be, we can always find some $M^\star \geq M$ and $\Pstar \in \Pcal$ so that $\sum_{k=0}^{M^\star-1} \Pstar(|Y|^q \1\{|Y|^q > M^\star \} > k) > \eps$. Writing out the above sum, we have for any $m, P$,
  \begin{align}
    \sum_{k=0}^{m-1} \Prob \left ( |Y|^q \1\{ |Y|^q > m\} > k \right )=  \sum_{k=0}^{m-1} \Prob \left ( |Y|^q > m \right )
  \end{align}
  since $|Y|^q \1\{|Y|^q > m\} > k$ if and only if $|Y|^q > m$ whenever $k \leq m$. Carrying on with the above calculation, we have as a consequence of \eqref{eq:limsup-of-lower-partial-sum} that
  \begin{equation}\label{eq:limsup-of-m-times-prob}
    \limsup_{m \to \infty}\sup_\Pin m \Prob (|Y|^q > m) > \eps.
  \end{equation}
  Simultaneously, by the assumed uniform tail vanishing property in the right-hand side of \eqref{eq:uniformly-integrable-qth-moment-tail-lemma}, we can choose an $M$ sufficiently large so that
  \begin{equation}\label{eq:sum-of-tail-probs-is-too-small}
    \sup_\Pin \sum_{k=M}^\infty \Prob \left ( |Y|^q \1 \{ |Y|^q > M \} > k \right ) < \eps / 2.
  \end{equation}
  By \eqref{eq:limsup-of-m-times-prob}, find some $M^\star > 3M$ and $\Pstar \in \Pcal$ so that
  \begin{equation}\label{eq:tail-prob-is-too-big}
    M^\star \cdot \Pstar \left ( |Y|^q > M^\star \right ) > \eps.
  \end{equation}
  We will now show that \eqref{eq:sum-of-tail-probs-is-too-small} and \eqref{eq:tail-prob-is-too-big} are incompatible, leading to a contradiction. Indeed,
  \begin{align}
    &\sup_\Pin \sum_{k=M}^\infty \Prob \left ( |Y|^q \1 \{ |Y|^q > M \} > k \right ) \\
    \geq\ &\sum_{k=M}^\infty \Pstar \left ( |Y|^q \1 \{ |Y|^q > M \} > k \right ) \\
    \geq\ &\sum_{k=M}^{M^\star} \Pstar \left ( |Y|^q \1 \{ |Y|^q > M \} > k \right ) \\
    =\ &\sum_{k=M}^{M^\star} \Pstar \left ( |Y|^q > k \right ),
  \end{align}
  where the first inequality follows by definition of a supremum, the second since we are taking only a smaller sum over finitely many elements, and the last equality follows from the fact that whenever $k \geq M$, $|Y|^q \1\{ |Y|^q > M\} > k$ if and only if $|Y|^q > k$. Carrying on with the above calculation, we finally have
  \begin{align}
    &\sup_\Pin \sum_{k=M}^\infty \Prob \left ( |Y|^q \1 \{ |Y|^q > M \} > k \right ) \\
    \geq\ &\sum_{k=M}^{M^\star} \Pstar \left ( |Y|^q > k \right ) \\
    \geq\ &\sum_{k=M}^{M^\star} \Pstar \left ( |Y|^q > M^\star \right ) \\
    =\ &(M^\star - M + 1) \Pstar \left ( |Y|^q > M^\star \right )\\
    \geq\ & (M^\star - M^\star / 3) \Pstar \left ( |Y|^q > M^\star \right )\\
    >\ & 2\eps / 3,
  \end{align}
  which contradicts the fact that the same sum was assumed to be less than $\eps / 2$ in \eqref{eq:sum-of-tail-probs-is-too-small}, completing the proof.
\end{proof}

\subsection{Proof details for \cref*{theorem:slln-noniid}}\label{proof:slln-noniid}

\begin{proof}[Proof of \cref{theorem:slln-noniid}]
  By \cref{lemma:sufficient-conditions-puniform-kolmogorov-three-series} below, we have that the right-hand side of \eqref{eq:slln-noniid-moment-condition}
implies that the three series in \cref{theorem:kolmogorov-three-series} vanish $\Pcal$-uniformly with $c= 1$ for the random variables $Z_k$ given by
\begin{equation}
  Z_k := \frac{X_k - \EE_P(X_k)}{a_k}.
\end{equation}
Therefore, by the $\Pcal$-uniform three series theorem (\cref{theorem:kolmogorov-three-series}) we have that $S_n \equiv S_n(P) := \sum_{k=1}^n (X_k - \EE_P(X_k)) / a_k$ forms a $\Pcal$-uniform Cauchy sequence.
Moreover, by \cref{theorem:boundedness three series theorem} and \cref{lemma:sufficient-conditions-puniform-boundedness-noniid}, we have that the $\Pcal$-uniform boundedness condition in the left-hand side of \eqref{eq:slln-noniid-moment-condition} implies that $S_n$ is $\Pcal$-uniformly stochastically bounded.
Combining the facts that $\infseqn{S_n}$ is $\Pcal$-uniformly Cauchy and stochastically bounded, we invoke the $\Pcal$-uniform Kronecker lemma (\cref{lemma:kronecker}) similar to the proof of \cref{theorem:mz-slln}(i) but now with the sequence $\infseqn{a_n}$ to yield the desired result, completing the proof of \cref{theorem:slln-noniid}.
\end{proof}

\begin{lemma}[Satisfying the three series for independent random variables]\label{lemma:sufficient-conditions-puniform-kolmogorov-three-series}
  Let $\infseqn{X_n}$ be independent random variables on $(\Omega, \Fcal, \Pcal)$ and let $Y_n := X_n - \EE X_n$ be their centered versions. Let $q \in [1,2]$. Suppose that for some positive and increasing sequence $(a_n)_{n=1}^\infty$ that diverges to $\infty$,
  \begin{equation}
    \lim_{m \to \infty}\sup_\Pin \sum_{n=m}^\infty  \frac{\EE_P |Y_n|^q}{a_n^q} = 0.
  \end{equation}
  Then the three series conditions of \cref{theorem:kolmogorov-three-series} are satisfied for $Z_k := Y_k/a_k$ with $c = 1$.
\end{lemma}

\begin{proof}
  First, define the truncated random variables $Z_n^\leqo := Z_n \1 \{|Z_n| \leq 1\}$
  for any $n$. We will satisfy the three series separately.
  \paragraph*{The first series} Writing out $\sup_\Pin \sum_{n=m}^\infty | \EE_P Z_n^\leqo |$, we have
  \begin{align}
    \sup_\Pin \sum_{n=m}^\infty  | \EE_P Z_n^\leqo | &= \sup_\Pin \sum_{n=m}^\infty \left \lvert  \EE_P \left ( \frac{Y_n\1 \{ |Y_n| \leq a_n \}}{a_n} \right ) \right \rvert\\
    \ifverbose %
    &= \sup_\Pin \sum_{n=m}^\infty \EE_P \left ( \frac{-Y_n\1 \{ |Y_n| \geq a_n \}}{a_n} \right )\\
    \fi %
    &\leq \sup_\Pin \sum_{n=m}^\infty \EE_P \left ( \frac{|Y_n|\1 \{ |Y_n| > a_n \}}{a_n} \right )\\
    &\leq \sup_\Pin \sum_{n=m}^\infty  \frac{\EE_P|Y_n|^q}{a_n^q} \to 0
  \end{align}
  where the last inequality follows from the fact that $(|Y_n|/a_n)\1 \{|Y_n| > a_n\} \leq (|Y_n|^q/a_n^q)\1\{ |Y_n| > a_n \}$ since $q \geq 1$.
  \paragraph*{The second series} Writing out $\sup_\Pin \sum_{n=m}^\infty \Var_P Z_n^\leqo$, we have
  \begin{align}
    \sup_\Pin \sum_{n=m}^\infty \Var_P Z_n^\leqo &\leq \sup_\Pin \sum_{n=m}^\infty \EE_P [ (Z_n^\leqo)^2] \\
    &= \sup_\Pin \sum_{n=m}^\infty \EE_P \left ( \frac{Y_n^2 \1 \{ |Y_n| \leq a_n \} }{a_n^2}\right ) \\
    &\leq \sup_\Pin \sum_{n=m}^\infty \frac{\EE_P |Y_n|^q}{a_n^q}  \to 0,
  \end{align}
  where the last inequality follows from the fact that $(Y_n^2 / a_n^2)\1\{ |Y_n| \leq a_n \} \leq (|Y_n|^q / a_n^q) \1\{ |Y_n| \leq a_n \}$ with $P$-probability one for each $P \in \Pcal$ since $q \leq 2$.
  \paragraph*{The third series} Finally, writing out the third series $\sup_\Pin \sum_{n=m}^\infty \Prob \left ( \left \lvert Y_n / a_n \right \rvert > 1 \right )$, we have by Markov's inequality,
  \begin{align}
    \sup_\Pin \sum_{n=m}^\infty \Prob \left ( \left \lvert \frac{Y_n}{a_n} \right \rvert > 1 \right )
    \leq \sup_\Pin \sum_{n=m}^\infty \frac{\EE |Y_n|^q}{a_n^q} \to 0, 
  \end{align}
  which completes the proof.
\end{proof}

\begin{lemma}[Satisfying the three series of $\Pcal$-uniform stochastic boundedness under independence]\label{lemma:sufficient-conditions-puniform-boundedness-noniid}
  Let $\infseqn{X_n}$ be independent random variables and let $Y_n := X_n - \EE X_n$ be their centered versions. Suppose that there exists $q \in [1,2]$ so that $\EE|Y_n|^q < \infty$ for each $n$ and that for some positive and increasing sequence $(a_n)_{n=1}^\infty$ that diverges to $\infty$,
  \begin{equation}\label{eq:stochastic-nonincreasingness-lemma-condition}
    \sup_\Pin \sum_{n=1}^\infty  \frac{\EE_P |Y_n|^q}{a_n^q} < \infty.
  \end{equation}
  Then the three series conditions of \cref{theorem:boundedness three series theorem} are satisfied for $Z_k := Y_k / a_k$.
\end{lemma}

\begin{proof}
  First, define the truncated random variables $Z_{n,B}^\leqo$ as
  \begin{equation}
    Z_{n,B}^\leqo := \frac{(Y_n / B) \1\{ |Y_n / B| \leq a_n \}}{a_n} .
  \end{equation}
  We will satisfy the three series separately.
  \paragraph*{The first series} Writing out $\sup_\Pin \sum_{n=1}^\infty  \left \lvert \EE_P Z_{n,B}^\leqo \right \rvert$ for any $B > 0$, we have
  \begin{align}
    \sup_\Pin \sum_{n=1}^\infty \left \lvert \EE_P Z_{n,B}^\leqo \right \rvert &= \sup_\Pin \sum_{n=1}^\infty \left \lvert \EE_P \left ( \frac{(Y_n / B)\1 \{ |Y_n/ B| \leq a_n \}}{a_n} \right ) \right \rvert\\
    &\leq \sup_\Pin \sum_{n=1}^\infty \EE_P \left ( \frac{|Y_n/ B|\1 \{ |Y_n/ B| > a_n \}}{a_n} \right )\\
    &\leq \frac{1}{B^q} \underbrace{\sup_\Pin \sum_{n=1}^\infty  \frac{\EE_P|Y_n|^q}{a_n^q}}_{< \infty},
  \end{align}
  and we note that finiteness of the supremum on the final line follows from the condition of the lemma in \eqref{eq:stochastic-nonincreasingness-lemma-condition}. Taking limits as $B \to \infty$ completes the argument for the first series.
  \paragraph*{The second series} Writing out $\sup_\Pin \sum_{n=1}^\infty \Var_P Z_{n,B}^\leqo$, we have
  \begin{align}
    \sup_\Pin \sum_{n=1}^\infty \Var_P Z_{n,B}^\leqo &\leq \sup_\Pin \sum_{n=1}^\infty \EE_P [ (Z_{n,B}^\leqo)^2] \\
    &= \sup_\Pin \sum_{n=1}^\infty \EE_P \left ( \frac{(Y_n/ B)^2 \1 \{ |Y_n/B| \leq a_n \} }{a_n^2}\right ) \\
    &\leq \frac{1}{B^q}\underbrace{\sup_\Pin \sum_{n=1}^\infty \frac{\EE_P |Y_n|^q}{a_n^q}}_{< \infty}.
  \end{align}
  Therefore, $\lim_B \sup_\Pin \sum_{n=1}^\infty \Var_P Z_{n,B}^\leqo = 0$, completing the proof for the second series. 
  \paragraph*{The third series} Finally, writing out the third series $\sup_\Pin \sum_{n=1}^\infty \Prob \left ( \left \lvert Y_n / B  \right \rvert > a_n \right )$, we have by Markov's inequality,
  \begin{align}
    \sup_\Pin \sum_{n=1}^\infty \Prob \left ( \left \lvert Y_n / B \right \rvert > a_n \right )
    \leq \frac{1}{B^q} \underbrace{\sup_\Pin \sum_{n=1}^\infty \frac{\EE |Y_n|^q}{a_n^q}}_{< \infty} \to 0 
  \end{align}
  as $B \to \infty$ which completes the proof.
\end{proof}
\section{Summary \& discussion}

In this paper, we introduced a set of tools and techniques to derive distribution-uniform strong laws of large numbers, culminating in extensions of Chung's \iid{} strong law to uniformly integrable $q^\tth$ moments for $0 < q < 2;\ q\neq 1$ in the sense of \citet{marcinkiewicz1937fonctions} as well as to independent but non-identically distributed random variables. Furthermore, we showed that $\Pcal$-uniform integrability of the $q^\tth$ moment is both sufficient and necessary for the strong law to hold at the Marcinkiewicz-Zygmund rates of $\oPcalas(n^{1/q-1})$, leading to a comprehensive understanding of uniform strong laws for sums of independent random variables.

We anticipate that the proof techniques found in Sections~\ref{section:other-strong-laws} and~\ref{section:complete proofs} may open vistas for understanding other distribution-uniform almost sure behavior. In particular, we plan to explore their use in the development of distribution-uniform analogues of strong invariance principles such as the Koml\'os-Major-Tusn\'ady embeddings \citep{komlos1975approximation,komlos1976approximation} in the presence of $q^\tth$ uniformly integrable moments when $q > 2$.

We also look forward to potential applications in distributionally robust optimization and portfolio theory where guarantees are often sought to hold in a worst-case sense over a family of distributions $\Pcal$; see \citep{bouchard2015arbitrage,obloj2021distributionally,blanchet2023statistical,blanchet2025distributionally}. We anticipate that if future extensions of these optimization routines involve the estimation of some unknown parameters from independent samples, distribution-uniform laws of large
numbers will find direct applications.

%% file: appendix.tex
\section{Auxiliary lemmas on uniform integrability}

Throughout the proofs of the main theorems, we relied on certain properties of functions that appear in de la Vall\'ee-Poussin's criterion on uniform integrability. This section collects and presents those properties as technical lemmas with self-contained proofs.

Beginning with \cref{lemma:canizo-vallee-poussin}, we provide a proof of de la Vall\'ee-Poussin's criterion where the slowly diverging function (denoted by $h$) can be taken to be concave, strictly increasing, and beginning at 0. The proof uses the same argument found in some notes on the website of Jos\'e A. Ca\~nizo \citep{canizo}. 

\begin{lemma}\label{lemma:canizo-vallee-poussin}
  Let $Y$ be a random variable on the probability spaces $(\Omega, \Fcal, \Pcal)$ so that $\EE_P |Y| <\infty$ for every $P \in \Pcal$. Then $Y$ is (uncentered) uniformly integrable if and only if there exists a measurable function $h : [0, \infty) \to [0, \infty)$ that is concave, strictly increasing, diverging to $\infty$, and starting at $h(0)=0$ so that
  \begin{equation}\label{eq:canizo-vallee-poussin}
    \sup_\Pin \EE_P \left [ |Y| h(|Y|) \right ] < \infty.
  \end{equation}
\end{lemma}

\begin{proof}
  The ``if'' direction follows from the usual statement of de la Vall\'ee-Poussin's theorem \citep{chong1979theorem} so we will focus on the ``only if'' direction. That is, we are tasked with finding a function $h$ satisfying the conditions above. To begin, define the uniform integrability tail $\UU_P(y)$ with lower truncation level $y \geq 0$ as
  \begin{equation}
    \UU_P(y) := \EE_P \left ( |Y| \1 \{ |Y| \geq y \} \right ).
  \end{equation}
  and let $\UU(y) := \sup_\Pin \UU_P(y)$, noting that $\UU$ is nonincreasing and $\UU(y) \to 0$ as $y \to \infty$ by $Y$ being uniformly integrable. Now, for every $n = 1, 2, \dots$, define
  \begin{equation}
    a_n := \inf\{ y > 0 : \UU(y) < 1 / n^2 \}
  \end{equation}
  and consider the increasing sequence $(x_n)_{n = 0, 1, \dots}$ starting at $x_0 = 0$ and defined recursively as
  \begin{equation}
    x_{n+1} := \max \{ x_n + 1, a_{n+1} + 1 \},
  \end{equation}
  noting that $x_n$ is strictly increasing in $n$ and diverging to $\infty$. Further, observe that $\UU(x_n) \leq 1/n^2$ since $x_n \geq a_n$ for each $n$. Now define the function $\phi$ as the step function given by
  \begin{equation}
    \phi(x) := n+1 \quad\text{when $x \in [x_n, x_{n+1})$ for $n = 0, 1,2 ,\dots$,}
  \end{equation}
  and notice that $\phi(x)$ diverges in $x$ because $(x_n)_n$ does in $n$. We now observe that $\supP \EE_P [ |Y| \phi(|Y|) ]$ is bounded:
  \begin{align}
    \supP \EE_P \left [ |Y| \phi(|Y|) \right ] &= \supP \EE_P \left [ \sum_{n=0}^\infty |Y| \1 \{ |Y|\geq x_n \} \right ] \\
    \ifverbose %
    &= \supP \EE_P \left [ \lim_{N \to \infty} \sum_{n=0}^N |Y| \1 \{ |Y|\geq x_n \} \right ] \\
    &= \supP \lim_{N \to \infty}\sum_{n=0}^N \EE_P \left [ |Y| \1 \{ |Y|\geq x_n \} \right ] \\
    \fi %
    &= \supP \sum_{n=0}^\infty \EE_P \left [ |Y| \1 \{ |Y|\geq x_n \} \right ] \\
    &= \supP \sum_{n=0}^\infty \UU_P ( x_n ) \\
                                               &\leq \supP \underbrace{\UU_P(0)}_{\EE_P |Y|} + \underbrace{\supP \sum_{n=1}^\infty \frac{1}{n^2}}_{\pi^2 / 6} < \infty.
  \end{align}
  Now, we construct $h$ as a piecewise-linear, nonnegative, and continuous concave minorant of $\phi$ which is also diverging to $\infty$. Let $(d_n)_{n=0}^\infty$ denote the piecewise derivatives of $h$ and define both $d_n$ and $h(x)$ recursively as $d_0 = 1$, $h(0) = 0$ and
  \begin{align}
    d_{n+1} &:= \min \left \{ d_{n}, \frac{n+1-h(x_n)}{x_{n+1} - x_n} \right \}; \quad n = 0, 1, 2, \dots,\\
    h(x) &:= h(x_n) + d_{n+1}(x-x_n); \quad x \in [x_n, x_{n+1}].
  \end{align}
  Clearly $h$ is continuous and strictly increasing. We have $h \leq \phi$ since for every $x \in [x_n, x_{n+1})$, $h(x) \leq h(x_{n+1}) \leq n+1 = \phi(x)$. Notice that $h$ diverges to $\infty$ since either $d_{n+1} = (n+1 - h(x_n)) / (x_{n+1}-x_n)$ infinitely often, in which case $h(x_n) = n$ infinitely often, or there exists some $d^\star$ for which $h(x) = h(x_n) + d^\star (x - x_n)$ for all $x \in [x_n,x_{n+1}]$ and for all $n$ sufficiently large. Moreover, notice that it is differentiable on each interval $(x_n, x_{n+1})$ with decreasing derivative $d_{n+1}$ so $h$ is concave. This completes the proof.
\end{proof}

The following lemma establishes an analogue of de la Vall\'ee-Poussin's criterion where the slowly diverging function enjoys certain polynomial-like properties. The two subsequent lemmas are used as technical devices in the proof of \cref{lemma:phi-inequality}.
\begin{lemma}\label{lemma:phi-inequality}
  Let $0 < q < p \leq 2$ and suppose that the (potentially uncentered) $q^\tth$ moment of $Y$ is UI, meaning that
  \begin{equation}
    \lim_\mto \sup_\Pin \EE_P \left ( |Y|^q \1(|Y|^q > m) \right ) = 0.
  \end{equation}
  Then there exists a function $\phi$ that can be written as $\phi(x) = x g(x)$ for any $x \geq 0$ where $g$ is diverging to $\infty$ and satisfies
  \begin{equation}
    \frac{a^{p/q}}{b^{p/q}}\leq \frac{a g(a)}{b g(b)}\quad\text{for all $0 \leq a \leq b$ and $b > 0$} ,
  \end{equation}
  and so that $\phi(|Y|^q)$ is $\Pcal$-uniformly bounded in expectation:
  \begin{equation}
    \sup_\Pin \EE_P \phi(|Y|^q) < \infty.
  \end{equation}
\end{lemma}

\begin{proof}
  By the criterion of uniform integrability due to Charles de la Vall\'ee Poussin \citep{chong1979theorem,hu2011note,chandra2015vallee}, we have that there exists a function $h : [0, \infty) \to [0, \infty)$ diverging to $\infty$ so that
  \begin{equation}
    \sup_\Pin \EE_P \left [ |Y|^q h(|Y|^q) \right ] < \infty.
  \end{equation}
  Moreover, $h$ can be assumed to be concave, strictly increasing, and starting at $h(0) = 0$ (see the Supplementary Material \citep{supplement}). Let $h^\star(x) := h(x) + 1$ for each $x$, and use $h^\star$ in \cref{lemma:upper bound by phi} to obtain a function $g$ that is concave, strictly increasing, and starting at $g(0) \geq 1$, and in addition, $g(x) \leq h(x) + 1$ for all $x \geq 0$ so that
  \begin{equation}
    \frac{a^{p/q}}{b^{p/q}} \leq \frac{a g(a)}{b g(b)}
  \end{equation}
  whenever $0 \leq a \leq b$ and $b > 0$. Noticing that since $g \leq h + 1$, we have
  \begin{equation}
    \sup_\Pin \EE_P \left [ |Y|^q g(|Y|^q) \right ] \leq \supP \EE_P \left [ |Y|^q \right ] + \supP \EE_P \left [ |Y|^q h(|Y|^q) \right ] < \infty.
  \end{equation}
  Taking $\phi(x) = x g(x)$ for $x \geq 0$ completes the proof.
\end{proof}

\begin{lemma}\label{lemma:upper bound by phi}
  Let $h^\star : \RR^{\geq 0} \to \RR^{\geq 0}$ be a function that is concave, strictly increasing, diverging to $\infty$, and beginning at $h^\star(0) \geq 1$. Let $0 < q < p \leq 2$. Then there exists a function $g$ (depending on $p$ and $q$) with all the aforementioned properties and in addition, $g(x) \leq h^\star(x)$ for all $x \geq 0$ so that for any real $a \geq 0$ and $b > 0$ such that $a \leq b$,
  \begin{equation}
    \frac{a^{p/q}}{b^{p/q}} \leq \frac{a g(a)}{bg (b)}.
  \end{equation}
\end{lemma}
\begin{proof}
  Throughout, denote $\delta := p/q -1$. Choose $g(x) := (h^\star(x))^{\delta}$ if $\delta \in (0, 1)$ and $g(x) := h^\star(x)$ if $\delta \geq 1$. Notice that since $h^\star(x) \geq 1$, we have that $g$ is also concave, strictly increasing, diverging to $\infty$, and beginning at $g(0) \geq 1$.
  Clearly the result of the lemma holds for this choice of $g$ when $a = 0$ since $g(0) \geq 1$ and $g$ is strictly increasing so let us focus on the case where $0 < a \leq b$. Showing the desired result is equivalent to showing that
  \begin{equation}
      a^{-\delta} g(a) \geq b^{-\delta} g(b) \quad\text{for all $a \leq b$}
  \end{equation}
  or in other words, that $x^{-\delta} g(x)$ is nonincreasing on $x > 0$. By \cref{lemma:nonincreasingness of h/x}, we have that $x^{-1}h^\star(x)$ is nonincreasing on $x > 0$. If $\delta \in (0, 1)$, then $x^{-\delta} g(x) \equiv  (x^{-1} h^\star(x) )^\delta$ is nonincreasing. On the other hand, if $\delta \geq 1$, then $x^{-\delta} g(x) = x^{1-\delta} (h^\star(x) / x)$ is the product of positive and nonincreasing functions and is hence nonincreasing. This completes the proof.
\end{proof}

\begin{lemma}\label{lemma:nonincreasingness of h/x}
  Let $h : \RR^{\geq 0} \to \RR^{\geq 0}$ be a function that is concave and beginning at $h(0) \geq 0$. Then, the function $h(x) / x$ 
  is nonincreasing on $x > 0$.
\end{lemma}
\begin{proof}
  Let $0 < a < b$. Our goal is to show that $h(a) / a \geq h(b) / b$. Indeed, notice that by concavity of $h$ and the fact that $h(0) \geq 0$, we have that
  \begin{align}
    h(a) = h\left( \left(1-\frac{a}{b}\right) \cdot 0 + \frac{a}{b} \cdot b \right) \geq
    \left(1-\frac{a}{b}\right) h(0) + \frac{a}{b} h(b) \geq
    \frac{a}{b} h(b).
  \end{align}
  Dividing by $a$ yields the desired result.
\end{proof}